\documentclass[11pt]{article}
\usepackage{amssymb}
\usepackage{amsmath}
\usepackage{mathrsfs}
\usepackage{graphics}
\usepackage{graphicx}
\usepackage{xcolor}
\usepackage{subfigure}
\usepackage[T1]{fontenc}
\usepackage{latexsym,amssymb,amsmath,amsfonts,amsthm}\usepackage{txfonts}
\newcommand{\be}{\begin{eqnarray}}
\newcommand{\ben}{\begin{eqnarray*}}
\newcommand{\en}{\end{eqnarray}}
\newcommand{\enn}{\end{eqnarray*}}

\newtheorem{theorem}{Theorem}[section]
\newtheorem{lemma}{Lemma}[section]
\newtheorem{prp}[theorem]{Proposition}
\newtheorem{thm}[theorem]{Theorem}
\newtheorem{cor}[theorem]{Corollary}
\newtheorem{dfn}{Definition}[section]

\newtheorem{remark}{Remark}

\definecolor{rr}{rgb}{1,0.5,0}
\begin{document}
\renewcommand{\theequation}{\arabic{section}.\arabic{equation}}
\begin{titlepage}
\title{\bf Semilinear stochastic partial differential equations: central limit theorem and moderate deviations
\thanks{This research is supported partially by National Natural Science Foundation of China (NSFC) (No. 11801032, 61873325, 11831010), Key Laboratory of Random Complex Structures and Data Science, Academy of Mathematics and Systems Science, Chinese Academy of Sciences(No. 2008DP173182), China Postdoctoral Science Foundation funded project (No. 2018M641204),  Southern University of Science and Technology Start up fund Y01286120.}}
\author{Rangrang Zhang$^{1,}$\thanks{Corresponding author},\ \   Jie Xiong$^{2}$\\
{\small $^1$  School of  Mathematics and Statistics,
Beijing Institute of Technology, Beijing, 100081, China}\\
{\small $^2$  Department of Mathematics, Southern University of Science and Technology, Shenzhen, 518055, China .}\\
( {\sf rrzhang@amss.ac.cn}\ \ {\sf xiongj@sustech.edu.cn})}
\date{}
\end{titlepage}
\maketitle

\noindent\textbf{Abstract}: In this paper, we establish a central limit theorem (CLT) and the moderate deviation principles (MDP) for a class of semilinear stochastic partial differential equations driven by multiplicative noise on a bounded domain. The main results can be applied to stochastic partial differential equations of various types such as the stochastic Burgers equation and the reaction-diffusion equations perturbed by space-time white noise.



\section{Introduction}
In this paper, we are concerned with the following semilinear stochastic partial differential equations (SPDE):
\begin{eqnarray}\label{e-1}
\frac{\partial u(t,x)}{\partial t}=\frac{\partial^2 u(t,x)}{\partial x^2}+b(t,x,u(t,x))+\frac{\partial g(t,x, u(t,x))}{\partial x}+\sigma(t,x,u(t,x))\frac{\partial^2}{\partial t\partial x}W(t,x)
\end{eqnarray}
with Dirichlet boundary condition
\begin{eqnarray*}
u(t,0)=u(t,1)=0, \quad t\in [0,T]
\end{eqnarray*}
and the initial condition
\begin{eqnarray*}
u(0,x)=f(x)\in L^2([0,1]),
\end{eqnarray*}
where $W(t,x)$ denotes the Brownian sheet on a filtered probability space $(\Omega, \mathcal{F}, \{\mathcal{F}_t\}, P)$ with expectation $E$. The functions $b=b(t,x,r)$, $g=g(t,x,r)$, $\sigma=\sigma(t,x,r)$ are Borel functions of $(t,x,r)\in \mathbb{R}^+\times [0,1]\times \mathbb{R}$. Linear growth on $b$ and quadratic growth on $g$ are assumed in Section \ref{sec-1}. Hence, the semilinear SPDE (\ref{e-1}) contains both the stochastic Burgers equation and the stochastic reaction-diffusion equations as special cases.
As a result, it attracts substantial research interests. There is an extensive literature about the semilinear SPDE (\ref{e-1}). For example, the existence and uniqueness of solutions to (\ref{e-1}) in the space $C([0,T]; L^2([0,1]))$ was studied by Gy\"{o}ngy in \cite{G98}. Foondun and Setayeshgar \cite{FS17} proved the large deviation principles (LDP) of the strong solution to (\ref{e-1}) holds uniformly on compact subsets of $C([0,T]; L^2([0,1]))$. Moreover, the ergodic theory of (\ref{e-1}) was studied by Dong and Zhang in \cite{DZ18}, where they show the existence and uniqueness of invariant measures of (\ref{e-1}). If the condition on $g$ is strengthen to be Lipschitz, Zhang \cite{Z18} proved Harnack inequalities for (\ref{e-1}) by using coupling method.

The purpose of this paper is to investigate deviations of the strong solution $u^\varepsilon$ of the semilinear SPDE (see (\ref{eqq-9})) from the solution $u^0$ of the deterministic equation (see (\ref{eqq-10})), as $\varepsilon$ decreases to 0.
That is, we seek the asymptotic behavior of the trajectory,
\[
X^\varepsilon(t)=\frac{1}{\sqrt{\varepsilon}\lambda(\varepsilon)}(u^\varepsilon-u^0)(t), \quad t\in[0,T],
\]
where $\lambda(\varepsilon)$ is some deviation scale which strongly influences the asymptotic behavior of $X^\varepsilon$. Concretely, three cases are involved:
\begin{description}
  \item[(1)] The case $\lambda(\varepsilon)=\frac{1}{\sqrt{\varepsilon}}$ provides LDP, which has been proved by \cite{FS17}.
  \item[(2)] The case $\lambda(\varepsilon)=1$  provides the central limit theorem (CLT). We will show that $X^\varepsilon$ converges to a solution of a stochastic equation, as $\varepsilon$ decrease to 0 in Section \ref{sec-2}.
  \item[(3)] To fill in the gap between the CLT scale ($\lambda(\varepsilon)=1$) and the large deviations scale ($\lambda(\varepsilon)=\frac{1}{\sqrt{\varepsilon}}$), we will study the so-called moderate deviation principle (MDP) in Section \ref{sec-3}. Here, the deviations scale satisfies
      \begin{eqnarray}\label{e-43}
      \lambda(\varepsilon)\rightarrow +\infty,\ \sqrt{\varepsilon}\lambda(\varepsilon)\rightarrow 0\quad {\rm{as}} \ \varepsilon\rightarrow 0.
      \end{eqnarray}
\end{description}

MDP arises in the theory of statistical inference naturally providing us with the rate of convergence and a useful method for constructing asymptotic confidence intervals ( see, e.g. \cite{E-1,I-K,K,GXZ} and references therein).
Similar to LDP, the proof of moderate deviations is mainly based on the weak convergence approach, which is introduced by Dupuis and Ellis in \cite{DE}. The key idea is to prove some variational representation formula about
the Laplace transform of bounded continuous functionals, which will lead to proving an equivalence between
the Laplace principle and LDP. In particular, for Brownian functionals, an elegant variational
representation formula has been established by Bou\'{e}, Dupuis \cite{MP} and Budhiraja, Dupuis \cite{BD}.

Up to now, there are a series of results about the central limit theorem and moderate deviations for fluid dynamics models driven by white noise in time. For example, Wang et al. \cite{WZZ} established the CLT and MDP for 2D Navier-Stokes equations driven by multiplicative Gaussian noise in the space $C([0,T];H)\cap L^2([0,T];V)$ and Zhang et al. \cite{ZZG} proved that such results hold for 2D primitive equations. Moreover,
Dong et al. \cite{DXZZ} proved the MDP for 2D Navier-Stokes equations driven by multiplicative L\'{e}vy noises in $D([0,T];H)\cap L^2([0,T];V)$.
However, there are few results on CLT and MDP for stochastic partial differential equations driven by space-time white noise. Recently, Belfadli et al. \cite{B-B-M} claimed moderate deviations for stochastic Burgers equation. However, we cannot adapt their method to our model since we do not see how to apply Burkholder-Davis-Gundy inequality to a stochastic integral of the form $\xi(s,x)\equiv\int^t_0\int f(s,t,x,y)W(ds dy)$ to get an estimate on the expectation of the supremum of $\xi$ in $(t,x)$ when the integrand $f$ depends on $t$ and does not have a semimartingale property with respect to parameter $x$.  When this paper was written, we  noticed the independent work by Hu et al. \cite{Hu} for the same model. However, how do they handle the afore mentioned difficulty  is not clear to us.

The purpose of this paper is two-fold. The first part is to show  $X^\varepsilon$ satisfies the CLT in probability in $C([0,T]; L^2([0,1]))$.
Compared with stochastic partial differential equations driven by white noise in time, there are some difficulties when dealing with such equations driven by space-time white noise. Most notably, as we already mentioned in the last paragraph, it is not trivial to obtain estimate of the expectation of the
supremum of the stochastic integral when the integrand also depends on the time parameter. More precisely, let $Z^{\varepsilon}(t,x)=\frac{u^{\varepsilon}(t,x)-u^{0}(t,x)}{\sqrt{\varepsilon}}-Y$, referring to (\ref{eee-21}), it satisfies
\begin{eqnarray*}\notag
Z^{\varepsilon}(t,x)&=&\int^t_0\int^1_0G_{t-s}(x,y)\Big(\frac{b(u^{\varepsilon})-b(u^0)}{\sqrt{\varepsilon}}-\partial_rb(u^0)Y\Big)dsdy\\
\notag
&&\ -\int^t_0\int^1_0\partial_yG_{t-s}(x,y)\Big(\frac{g(u^{\varepsilon})-g(u^0)}{\sqrt{\varepsilon}}-\partial_rg(u^0)Y\Big)dsdy\\
\notag
&&\ +\int^t_0\int^1_0G_{t-s}(x,y)\Big(\sigma(s,y,u^{\varepsilon}(s,y))-\sigma(s,y,u^0(s,y))\Big)W(dyds)\\
&:=& I^{\varepsilon}_1(t,x)+I^{\varepsilon}_2(t,x)+I^{\varepsilon}_3(t,x).
\end{eqnarray*}
Our aim is to prove $\sup_{t\in [0,T]}\|Z^{\varepsilon}(t)\|^2_{L^2([0,1])}$ converges to $0$ in probability, so we need to show $\sup_{t\in [0,T]}\|I^{\varepsilon}_i(t)\|^2_{L^2([0,1])}$ converges to $0$ in probability for $i=1,2,3$. Since either $G_{t-s}(x,y)$  or $\partial_yG_{t-s}(x,y)$ is contained in $I^{\varepsilon}_i(t,x)$ and they are both not increasing with respect to $t$, we can not take supremum of $t\in [0,T]$ directly in front of $\|I^{\varepsilon}_i(t,x)\|^2_{L^2([0,1])}$, $i=1,2$. In particular, to estimate
$E\sup_{t\in [0,T]}\|I^{\varepsilon}_3(t)\|_{L^2([0,1])}$,
 the Burkholder-Davis-Gundy inequality is not applicable because the dependence of the integrand on $t$. To overcome this difficulty, we employ the Garsia lemma from \cite{Xiong}, which gives a way to make estimates of $\sup_{t\in [0,T]}\|I^{\varepsilon}_i(t)\|^2_{L^2([0,1])}$. However, it requires an appropriate continuity property of $\|I^{\varepsilon}_i(t)\|^2_{L^2([0,1])}$ with respect to $t$. In order to achieve this condition, some delicate a priori estimates are necessary (see Section \ref{sec-4}).
The second part is to prove MDP for $X^\varepsilon$ in the space $C([0,T]; L^2([0,1]))$, which is equivalent  to proving that $X^\varepsilon$ satisfies a large deviation principle in $C([0,T]; L^2([0,1]))$ with $\lambda(\varepsilon)$ satisfying (\ref{e-43}). The proof will be based on the weak convergence approach introduced by Bou\'{e} and Dupuis \cite{MP}, Budhiraja and Dupuis \cite{BD}. Except difficulties mentioned above, the proof of some tightness results in $C([0,T]; L^2([0,1]))$  are also nontrivial.

\

This paper is organized as follows. The mathematical formulation of the semilinear stochastic partial differential equations is presented in Sect. 2. Some delicate a priori estimates are given in Sect. 3. In Sect. 4, the central limit theorem is established. Finally, the moderate deviation principles is proved in Sect. 5.

Throughout the whole paper, the constant $C$ is different from line to line.
\section{Framework}
Let $L^p([0,1]), p\in(0,\infty]$ be the Lebesgue space, whose norm is denoted by $\|\cdot\|_{L^p}$.
In particular, denote that $H=L^2([0,1])$ with the corresponding norm  $\|\cdot\|_H$ and inner product $(\cdot,\cdot)$.
Define an operator $A:= \frac{\partial^2}{\partial x^2}$. Let $G_{t}(x,y)=G(t,x,y), t\geq0, x,y\in [0,1]$ be the Green function for the operator $\partial_t-A$  with Dirichlet boundary condition. Then, it satisfies that
\begin{eqnarray}\label{e-29}
\partial_t G_{t}(x,y)=AG_{t}(x,y).
\end{eqnarray}
\subsection{Assumptions}\label{sec-1}
We adopt assumptions from  \cite{FS17} or \cite{G98}.
The functions $b=b(t,x,r)$, $g=g(t,x,r)$, $\sigma=\sigma(t,x,r)$ are Borel measurable on $(t,x,r)\in \mathbb{R}^+\times [0,1]\times \mathbb{R}$ and  satisfy the following conditions:
\begin{description}
  \item[(H1)]  $b$  is of linear growth, $g$ is of quadratic growth and $\sigma$ is bounded. That is, there exists a constant $K>0$ such that for all $(t,x,r)\in[0,T]\times[0,1]\times \mathbb{R}$, we have
      \begin{equation}
 |b(t,x,r)|\leq K(1+|r|),\quad |\sigma(t,x,r)|\leq K,
\end{equation}
and
 \begin{equation}
       |g(t,x,r)|\leq K(1+|r|^2).
    \end{equation}

  \item[(H2)] $\sigma$ is Lipschitz, $b$ and $g$ are locally Lipschitz with linearly growing Lipschitz constant. That is, there exists a constant $L$ such that for all $(t,x,r_1,r_2)\in[0,T]\times[0,1]\times \mathbb{R}^2$, we have
      \begin{eqnarray*}
      |b(t,x,r_1)-b(t,x,r_2)|&\leq& L(1+|r_1|+|r_2|)|r_1-r_2|,\\
      |g(t,x,r_1)-g(t,x,r_2)|&\leq& L(1+|r_1|+|r_2|)|r_1-r_2|,\\
      |\sigma(t,x,r_1)-\sigma(t,x,r_2)|&\leq& L|r_1-r_2|.
      \end{eqnarray*}
\end{description}
\begin{dfn}\label{dfn-1}
 A random field $u$ is a solution to (\ref{e-1}) if $u=\{u(t,\cdot), t\in \mathbb{R}^+\}$ is an $H-$valued continuous $\mathcal{F}_t-$adapted random field with initial value $f\in H$ and satisfying for all $t\geq 0$, $\phi\in C^{2}([0,1])$ with $\phi(0)=0$, $\phi(1)=0$,
\begin{eqnarray}\notag
&&\int^1_0u(t,x)\phi(x)dx=\int^1_0f(x)\phi(x)dx+\int^t_0\int^1_0 u(s,x)\frac{\partial^2 \phi(x)}{\partial x^2}dxds
+\int^t_0\int^1_0b(s,x, u(s,x))\phi(x)dxds\\
\label{e-2}
&& \quad -\int^t_0\int^1_0g(s,x,u(s,x))\frac{\partial \phi(x)}{\partial x}dxds+\int^t_0\int^1_0\sigma(s,x,u(s,x))\phi(x)W(dxds), \ P-a.s.
\end{eqnarray}

\end{dfn}

The existence and uniqueness of the solution to (\ref{e-1}) is established in \cite{G98}.
\begin{thm}\label{thm-1}
Under assumptions (H1)-(H2), there exists a unique solution $u$ to SPDE  (\ref{e-1}).
\end{thm}

\begin{remark}
Referring to Proposition 3.5 in \cite{G98}, under conditions in Theorem \ref{thm-1}, (\ref{e-2}) is equivalent to the following form: for all $t\geq 0$ and almost surely $\omega\in \Omega$,
 \begin{eqnarray}
&&u(t,x)=\int^1_0 G_t(x,y)f(y)dy +\int^t_0\int^1_0 G_{t-s}(x,y)b(s,y,u(s,y))dyds\\
\label{e-3}
&&\quad \quad  -\int^t_0\int^1_0 \partial_y G_{t-s}(x,y)g(s,y,u(s,y))dyds+\int^t_0\int^1_0G_{t-s}(x,y)\sigma(s,y,u(s,y))W(dyds)\notag
\end{eqnarray}
for almost every $x\in [0,1]$.
\end{remark}
In order to establish the CLT and MDP for (\ref{e-1}), we need some additional conditions on $b$ and $g$.
\begin{description}
  \item[(H3)] The  partial derivatives of $b$ and $g$ in $r$ are both of linear growth and Lipschitz. There exists a constant $K$ such that for any $(t,x,r)\in[0,T]\times[0,1]\times \mathbb{R}$ ,
\begin{eqnarray}\label{eqqq-2}
|\partial_rb(t,x,r)|\leq K(1+|r|),\quad |\partial_rg(t,x,r)|\leq K(1+|r|),\quad |\partial^2_rg(t,x,r)|\leq K,
\end{eqnarray}
and there exists a constant $L>0$ such that for all  $(t,x,r_1,r_2)\in[0,T]\times[0,1]\times \mathbb{R}^2$, we have
\begin{eqnarray}\label{eqqq-1}
|\partial_rb(t,x,r_1)-\partial_rb(t,x,r_2)|\leq L|r_1-r_2|,\quad |\partial_rf(t,x,r_1)-\partial_rf(t,x,r_2)|\leq L|r_1-r_2|.
\end{eqnarray}

For simplicity, we assume constants $K, L$ in (H3) are the same with those in (H1)-(H2).
\end{description}

\subsection{Properties of Green functions}
Referring to \cite{Wu}, the following facts will be used throughout this article:
\begin{description}
  \item[(1)] $\int_{\mathbb{R}}G_t(x,y)dy=1, \quad \int_{\mathbb{R}}G^2_t(x,y)dy=(2\pi t)^{-\frac{1}{2}}, \quad \forall t\in [0, \infty),  \forall x\in \mathbb{R}.$
  \item[(2)] $G_t(x,y)=G_t(y,x)$, \quad $t\in [0, \infty),\ x, y\in \mathbb{R}.$
  \item[(3)] $\int^1_0G_{t}(x,y)G_{s}(y,z)dy=G_{t+s}(x,z),\quad {\rm{for}}\ t,s\geq0,\  x,y, z\in [0,1].$
  \item[(4)] For any $m\leq 1, n\leq 2$, there exist $C,\tilde{C}>0$ such that
  \begin{eqnarray}\label{eq-21}
  \Big|\frac{\partial^m}{\partial t^m}\frac{\partial^n}{\partial y^n}G_t(x,y)\Big|\leq Ct^{-\frac{1+2m+n}{2}}e^{-\tilde{C}\frac{(x-y)^2}{t}}, \quad \forall t\in (0, \infty),\  \forall x,y\in \mathbb{R}.
  \end{eqnarray}
\end{description}
A particular case for \textbf{(4)} is $m=0, n=1$, in this case, we get
\begin{eqnarray}\label{eqq-4}
\partial_y G_t(x,y)\leq Ct^{-1}.
\end{eqnarray}
Referring to (3.13) in \cite{Walsh}, for $0<r<3$, it holds that
\begin{eqnarray}\label{eqq-5}
\int^1_0G^r_{t}(x,y)dy\leq Ce^{-tr}t^{\frac{1}{2}-\frac{1}{2}r}\leq Ct^{\frac{1}{2}-\frac{1}{2}r}.
\end{eqnarray}
Based on \textbf{(4)},  we deduce that for $0<r<\frac{3}{2}$, it holds that
\begin{eqnarray}\label{eqq-6}
\int^1_0|\partial_yG_t(x,y)|^rdy\leq Ct^{\frac{1}{2}-r}.
\end{eqnarray}
Moreover,
\begin{eqnarray}\label{eqq-5-1}
\sup_{x\in [0,1]}\int^s_0\int^1_0|G_{t-u}(x,y)-G_{s-u}(x,y))|^rdydu\leq  C|t-s|^{\frac{3-r}{2}}, \ 1<r<3.
\end{eqnarray}
and
\begin{eqnarray}\label{eq-29}
\sup_{x\in [0,1]}\int^s_0\int^1_0|\partial_yG_{t-u}(x,y)-\partial_yG_{s-u}(x,y))|^rdydu\leq  C|t-s|^{\frac{3}{2}-r}, \ 1<r<\frac{3}{2}.
\end{eqnarray}

For a transition kernel $H(r,t;x,y)$, we define the linear operator $J$ by
\begin{eqnarray}\label{e-15}
J(v)(t,x)=\int^t_0\int^1_0H(r,t;x,y)v(r,y)dydr, \ t\in [0,T],\ x\in[0,1]
\end{eqnarray}
for every $v\in L^{\infty}([0,T];L^1([0,1]))$.

Referring to \cite{G98}, we have the following heat kernel estimate, which is very crucial to our proof.
\begin{lemma}\label{lem-1}
Let $J$ is defined by $H(s,t;x,y)=G_{t-s}(x,y)$ or by $H(s,t;x,y)=\partial_y G_{t-s}(x,y)$ in (\ref{e-15}). Let $\rho\in[1,\infty]$, $q\in[1,\rho)$ and set $\kappa=1+\frac{1}{\rho}-\frac{1}{q}$. Then $J$ is a bounded linear operator from $L^{\gamma}([0,T];L^q([0,1]))$ into $C([0,T];L^{\rho}([0,1]))$ for $\gamma>2\kappa^{-1}$. Moreover, for any $T\geq 0$, there is $C>0$ such that
\begin{eqnarray}\label{e-16}
\|J(v)(t,\cdot)\|_{L^{\rho}}\leq C\int^t_0(t-s)^{\frac{\kappa}{2}-1}\|v(s,\cdot)\|_{L^q}ds.
\end{eqnarray}
\end{lemma}
In particular, taking $\rho=2, \kappa=\frac{1}{2}, q=1$, we deduce that
\begin{eqnarray}
\|J(v)(t,\cdot)\|_{L^{2}}\leq C\int^t_0(t-s)^{-\frac{3}{4}}\|v(s,\cdot)\|_{L^1}ds.
\end{eqnarray}
At last, we recall the following Garsia lemma from Lemma 10.2.1 in \cite{Xiong}, which plays a key role in this article.
\begin{lemma}\label{lem-3}
Let $(Z,d)$ be a metric space and let $\psi$ be a continuous map from $[0,T]$ to $Z$. Suppose that $\Psi$ and $p$ are increasing functions such that $\Psi(0)=p(0)=0$ and $\Psi$ is convex. Let
\[
\rho=\int^T_0\int^T_0\Psi\Big(\frac{d(\psi(t),\psi(s))}{p(|t-s|)}\Big)dtds
\]
Then, for any $t,s\in [0,T]$, we have
\begin{eqnarray*}
d(\psi(t), \psi(s))\leq 8\int^{|t-s|}_0 \Psi^{-1}(\rho r^{-2})dp(r),
\end{eqnarray*}
where $\Psi^{-1}$ denotes the inverse function of  $\Psi$.
\end{lemma}

\section{CLT for semilinear SPDE}\label{sec-2}
Let $u^{\varepsilon}(t,x)$ be the solution of the following equation
\begin{eqnarray}\notag
u^{\varepsilon}(t,x)&=&\int^1_0 G_t(x,y)f(y)dy +\int^t_0\int^1_0 G_{t-s}(x,y)b(s,y,u^{\varepsilon}(s,y))dyds\\
\notag
&&\  -\int^t_0\int^1_0 \partial_y G_{t-s}(x,y)g(s,y,u^{\varepsilon}(s,y))dyds\\
\label{eqq-9}
&&\ +\sqrt{\varepsilon}\int^t_0\int^1_0G_{t-s}(x,y)\sigma(s,y,u^{\varepsilon}(s,y))W(dyds).
\end{eqnarray}
Using the same method as Theorem 2.1 in \cite{G98}, we know that $\sup_{t\in [0,T]}\|u^{\varepsilon}(t)\|^2_H$ is bounded in probability, i.e.,
\begin{eqnarray}\label{eqq-13}
\lim_{C\rightarrow \infty}\sup_{0<\varepsilon\leq 1}P\Big(\sup_{t\in [0,T]}\|u^{\varepsilon}(t)\|^2_H>C\Big)=0.
\end{eqnarray}
Taking $\varepsilon\rightarrow 0$, it yields that
\begin{eqnarray}\notag
u^{0}(t,x)&=&\int^1_0 G_t(x,y)f(y)dy +\int^t_0\int^1_0 G_{t-s}(x,y)b(s,y,u^{0}(s,y))dyds\\
\label{eqq-10}
&&\  -\int^t_0\int^1_0 \partial_y G_{t-s}(x,y)g(s,y,u^{0}(s,y))dyds.
\end{eqnarray}
Define $Y^{\varepsilon}(t,x)=\frac{u^{\varepsilon}(t,x)-u^{0}(t,x)}{\sqrt{\varepsilon}}$, then $Y^{\varepsilon}$ satisfies
\begin{eqnarray}\notag
Y^{\varepsilon}(t,x)&=&\frac{1}{\sqrt{\varepsilon}}\int^t_0\int^1_0 G_{t-s}(x,y)\Big(b(s,y,u^{\varepsilon}(s,y))-b(s,y,u^{0}(s,y))\Big)dyds\\
\notag
&&\  -\frac{1}{\sqrt{\varepsilon}}\int^t_0\int^1_0 \partial_y G_{t-s}(x,y)\Big(g(s,y,u^{\varepsilon}(s,y))-g(s,y,u^{0}(s,y))\Big)dyds\\
\label{eqq-11}
&&\ +\int^t_0\int^1_0G_{t-s}(x,y)\sigma(s,y,u^{\varepsilon}(s,y))W(dyds).
\end{eqnarray}
Let $Y$ is the solution of the following equation
\begin{eqnarray}\notag
Y(t,x)&=&\int^t_0\int^1_0G_{t-s}(x,y)\partial_rb(s, y, u^0(s,y))Y(s,y)dsdy\\ \notag
&&\ -\int^t_0\int^1_0\partial_yG_{t-s}(x,y)\partial_rg(s, y, u^0(s,y))Y(s,y)dsdy\\
\label{eqq-12}
&&\ +\int^t_0\int^1_0G_{t-s}(x,y)\sigma(s,y, u^0(s,y))W(dsdy).
\end{eqnarray}

The first result of this article reads as
\begin{thm}\label{thm-2}
(Central Limit Theorem) Let the initial value $f\in L^p([0,1])$ for all $p\in [2,\infty)$. Under (H1)-(H3), $Y^{\varepsilon}(t)-Y\rightarrow 0$ in probability in $C([0,T]; H)$, i.e., for any $\delta>0$,
\begin{eqnarray*}
\lim_{\varepsilon\rightarrow 0}P\left(\sup_{t\in [0,T]}\|Y^{\varepsilon}(t)-Y(t)\|_H>\delta\right)=0.
\end{eqnarray*}
\end{thm}
\subsection{A priori estimates}\label{sec-4}

In order to establish CLT and MDP for semilinear  SPDE (\ref{e-1}), we need to make some delicate a priori estimates.
Let us start with $u^0$.
\begin{lemma}\label{lem-5}
Under (H1), there exists $C_0:=C(K,T)(1+\|f\|^2_H)$ such that
\begin{eqnarray*}
\sup_{t\in [0,T]}\|u^0(t)\|^2_H\leq C_0.
\end{eqnarray*}
\end{lemma}
\begin{proof}
For any $t\in [0,T]$, from (\ref{eqq-10}), we get
\begin{eqnarray*}
\frac{\partial u^0(t,x)}{\partial t}=\frac{\partial^2 u^0(t,x)}{\partial x^2}+b(t,x,u^0(t,x))+\partial_x g(t,x,u^0(t,x)).
\end{eqnarray*}
Utilizing the chain rule, it follows that
\begin{eqnarray*}
&&\|u^0(t)\|^2_H+2\int^t_0\|\partial_xu^0(s)\|^2_Hds\\
&=&\|f\|^2_H+2\int^t_0(u^0(s), b(s,u^0(s)))ds+2\int^t_0(u^0(s), \partial_x g(s,u^0(s)))ds\\
&=:& \|f\|^2_H+I_1(t)+I_2(t),
\end{eqnarray*}
By (H1), we have
\begin{eqnarray*}
I_1(t)\leq K\int^t_0\|u^0\|_H(1+\|u^0\|_H)ds\leq CKT+CK\int^t_0\|u^0(s)\|^2_Hds.
\end{eqnarray*}
By integration by parts, we have
\begin{eqnarray*}
I_2(t)=-2\int^t_0(\partial_x u^0(s), g(s,u^0(s)))ds.
\end{eqnarray*}
Let $h(t,r)=\int^r_0 g(t,z)dz, t\in [0,T], r\in \mathbb{R}$, by the boundary conditions, it follows that
 \begin{eqnarray*}
-2\int^t_0(\partial_x u^0(s), g(s,u^0(s)))ds=-2\int^t_0\int^1_0\frac{\partial}{\partial_x}h(s,u^0(s,x))dxds=0.
\end{eqnarray*}
Combining all the above estimates, we obtain
\begin{eqnarray*}
\|u^0(t)\|^2_H+\int^t_0\|\partial_xu^0(s)\|^2_Hds\leq  \|f\|^2_H+CKT+CK\int^t_0\|u^0(s)\|^2_Hds.
\end{eqnarray*}
By Gronwall inequality, we obtain the desired result.
\end{proof}
%

For any $0<\varepsilon\leq 1$ and $R>0$, define a stopping time
\begin{eqnarray}\label{eee-30}
\tau^{\varepsilon,R}:=\inf\{t\wedge T:\|u^{\varepsilon}(t)\|_H>R\}.
\end{eqnarray}
For simplicity, in the rest part, we denote that $\tau:=\tau^{\varepsilon,R}$.

Now, we make estimates of the difference between $u^{\varepsilon}$ and $u^0$, which is crucial to our proof of CLT for semilinear SPDE (\ref{e-1}).
\begin{lemma}\label{lem-2}
For any $R>0, p>8$, there exists $C_1=C(R,K,L,p,T, C_0)$ such that
\begin{eqnarray}\label{eqq-1}
\sup_{t\in [0,T]}E\int^1_0|u^{\varepsilon}(t\wedge\tau,x)-u^0(t\wedge\tau,x)|^pdx\leq\varepsilon^{\frac{p}{2}}C_1.
\end{eqnarray}
\end{lemma}
\begin{proof}
We deduce from (\ref{eqq-9}) and (\ref{eqq-10}) that
\begin{eqnarray*}
u^{\varepsilon}(t\wedge\tau,x)-u^0(t\wedge\tau,x)&=&\int^{t}_0\int^1_0G_{t\wedge \tau-s}(x,y)(b(u^{\varepsilon})-b(u^0))I_{\{s\leq \tau\}}dsdy\\
&&\ -\int^t_0\int^1_0\partial_yG_{t\wedge\tau-s}(x,y)(g(u^{\varepsilon}(s))-g(u^0(s)))I_{\{s\leq \tau\}}dsdy\\
&&\ +\sqrt{\varepsilon}\int^t_0\int^1_0G_{t\wedge\tau-s}(x,y)\sigma(s,y,u^{\varepsilon}(s,y))I_{\{s\leq \tau\}}W(dyds)\\
&:=& K^{\varepsilon}_1(t,x)+K^{\varepsilon}_2(t,x)+K^{\varepsilon}_3(t,x).
\end{eqnarray*}
By (H2) and H\"{o}lder inequality, we deduce that
\begin{eqnarray*}
|K^{\varepsilon}_1(t,x)|^p
&\leq & L^p\Big|\int^{t}_0\int^1_0G_{t\wedge\tau-s}(x,y)(1+|u^{\varepsilon}|+|u^0|)|u^{\varepsilon}-u^0|I_{\{s\leq \tau\}}dsdy\Big|^p\\
&\leq & L^p\Big|\int^{t}_0\Big[\int^1_0(1+|u^{\varepsilon}(s\wedge \tau)|^2+|u^0(s\wedge \tau)|^2)I_{\{s\leq \tau\}}dy\Big]^{\frac{1}{2}}\Big[\int^1_0G^2_{t\wedge \tau-s}(x,y)|u^{\varepsilon}-u^0|^2I_{\{s\leq \tau\}}dy\Big]^{\frac{1}{2}}ds\Big|^p\\
&\leq& L^p(1+R^2+C_0)^{\frac{p}{2}}\Big|\int^{t}_0\Big[\int^1_0G^2_{t\wedge \tau-s}(x,y)|u^{\varepsilon}(s)-u^0(s)|^2dy\Big]^{\frac{1}{2}}I_{\{s\leq \tau\}}ds|^p\\
&\leq& L^p(1+R^2+C_0)^{\frac{p}{2}} t^{\frac{p}{2}}\Big|\int^{t}_0\int^1_0G^2_{t\wedge \tau-s}(x,y)|u^{\varepsilon}(s\wedge  \tau)-u^0(s\wedge  \tau)|^2dyI_{\{s\leq \tau\}}ds\Big|^{\frac{p}{2}}\\
&\leq& L^p(1+R^2+C_0)^{\frac{p}{2}} t^{\frac{p}{2}}\Big|\Big(\int^{t\wedge \tau}_0\int^1_0G^{2 q}_{t\wedge \tau-s}(x,y)dyds\Big)^{\frac{p}{2q}}\\
&& \times \Big(\int^{t\wedge \tau}_0\int^1_0|u^{\varepsilon}(s\wedge  \tau)-u^0(s\wedge  \tau)|^pdyds\Big)\Big|,
\end{eqnarray*}
where $\frac{2}{p}+\frac{1}{q}=1$.

As $p>8$, we have $q=(1-2p^{-1})^{-1}<\frac{4}{3}<\frac32$, then
$
2 q<3.
$
It follows from (\ref{eqq-5}) that
\begin{eqnarray*}
|K^{\varepsilon}_1(t,x)|^p
&\leq& L^p(1+R^2+C_0)^{\frac{p}{2}}C(p,T)\int^{t\wedge \tau}_0\int^1_0|u^{\varepsilon}(s\wedge \tau,y)-u^0(s\wedge \tau,y)|^pdyds.
\end{eqnarray*}
By (\ref{eqq-4}) and H\"{o}lder inequality, for any $0<\delta<1$, we get
\begin{eqnarray*}
|K^{\varepsilon}_2(t,x)|^p&=& \Big|\int^t_0\int^1_0\partial_yG_{t\wedge \tau-s}(x,y)(g(u^{\varepsilon})-g(u^{0}))I_{\{s\leq \tau\}}dsdy\Big|^p\\
&\leq& L^p\Big|\int^{t}_0\big[\int^1_0|\partial_yG_{t\wedge \tau-s}(x,y)|^{2\delta}(1+|u^{\varepsilon}(s\wedge \tau)|^2+|u^0(s\wedge \tau)|^2)dy\big]^{\frac{1}{2}}\\
&& \times \big[\int^1_0|\partial_yG_{t\wedge \tau-s}(x,y)|^{2(1-\delta)}|u^{\varepsilon}(s\wedge \tau)-u^0(s\wedge \tau)|^2dy\big]^{\frac{1}{2}}I_{\{s\leq \tau\}}ds\Big|^p\\
&\leq& L^p\Big|\int^{t}_0(t\wedge \tau-s)^{-\delta}(1+\|u^{\varepsilon}(s\wedge \tau)\|^2_H+\|u^0(s\wedge \tau)\|^2_H)^{\frac{1}{2}}\\
&& \times \big[\int^1_0|\partial_yG_{t\wedge \tau-s}(x,y)|^{2(1-\delta)}|u^{\varepsilon}(s\wedge \tau)-u^0(s\wedge \tau)|^2 dy\big]^{\frac{1}{2}}I_{\{s\leq \tau\}}ds\Big|^p\\
&\leq& L^p(1+R^2+C_0)^{\frac{p}{2}}\Big|\int^{t}_0(t\wedge \tau-s)^{-\delta}\Big[\int^1_0|\partial_yG_{t\wedge \tau-s}(x,y)|^{2(1-\delta)}|u^{\varepsilon}-u^0|^2dy\Big]^{\frac{1}{2}}I_{\{s\leq \tau\}}ds\Big|^p\\
&\leq& L^p(1+R^2+C_0)^{\frac{p}{2}}\Big(\int^{t\wedge \tau}_0(t\wedge \tau-s)^{-2\delta}ds\Big)^{\frac{p}{2}}\\
&& \times \Big(\int^{t\wedge \tau}_0\int^1_0|\partial_yG_{t\wedge \tau-s}(x,y)|^{2(1-\delta)}|u^{\varepsilon}(s\wedge \tau)-u^0(s\wedge \tau)|^2dyds\Big)^{\frac{p}{2}}\\
&\leq& L^p(1+R^2+C_0)^{\frac{p}{2}}\Big(\int^{t\wedge \tau}_0(t\wedge \tau-s)^{-2\delta}ds\Big)^{\frac{p}{2}}\Big(\int^{t\wedge \tau}_0\int^1_0|\partial_yG_{t\wedge \tau-s}(x,y)|^{2(1-\delta) q}dyds\Big)^{\frac{p}{2q}}\\
&& \times \int^{t\wedge \tau}_0\int^1_0|u^{\varepsilon}(s\wedge \tau)-u^0(s\wedge \tau)|^p dyds,
\end{eqnarray*}
where $\frac{2}{p}+\frac{1}{q}=1$.

As $p>8$, we have $q=(1-2p^{-1})^{-1}<\frac{4}{3}$. Taking $\delta=\frac{15}{32}$, then
\[
-2 \delta>-1, \quad 0<2(1-\delta) q<\frac{3}{2}.
\]
With the aid of (\ref{eqq-6}), it follows that
\begin{eqnarray*}
&&|K^{\varepsilon}_2(t,x)|^p\\
&\leq& L^p(1+R^2+C_0)^{\frac{p}{2}}C(p,T)
\int^{t\wedge \tau}_0\int^1_0|u^{\varepsilon}(s\wedge \tau,y)-u^0(s\wedge \tau,y)|^p dyds.
\end{eqnarray*}
Finally, we estimate $K^{\varepsilon}_3(t,x)$. Define
\[
J(t,x)=\int^t_0\int^1_0G_{t-s}(x,y)\sigma(s,y,u^{\varepsilon}(s,y))I_{\{s\leq \tau\}}W(dyds).
\]
Then
\begin{eqnarray}\label{eq-31}
K^{\varepsilon}_3(t,x)=\sqrt{\varepsilon}J(t\wedge\tau,x).
\end{eqnarray}
Note that for any $0\le s<t\le T, x\in [0,1]$, by Burkholder-Davis-Gundy inequality, (H1), (\ref{eqq-5})  and (\ref{eqq-5-1}), we obtain
\begin{eqnarray}\notag
E|J(t,x)-J(s,x)|^p&\le&E\left|\int^s_0\int^1_0(G_{t-r}(x,y)-G_{s-r}(x,y))\sigma(r,y,u^{\varepsilon}(r,y))I_{\{r\leq \tau\}}W(dydr)\right|^p\\ \notag
&&+E\left|\int^t_s\int^1_0G_{t-r}(x,y)\sigma(r,y,u^{\varepsilon}(r,y))I_{\{r\leq \tau\}}W(dydr)\right|^p\\ \notag
&\le&K^pE\left|\int^s_0\int^1_0(G_{t-r}(x,y)-G_{s-r}(x,y))^2dydr\right|^{p/2}\\ \notag
&&+K^pE\left|\int^t_s\int^1_0G_{t-r}(x,y)^2dydr\right|^{p/2}\\
\label{eq-17}
&\le&K^p|t-s|^{\frac{p}{4}}.
\end{eqnarray}
Let
\[
\Psi(r)=r^p,\quad p(r)=r^{\frac{1}{4}},
\]
and
\[
\rho(x)=\int^T_0\int^T_0\left|\frac{|J(t,x)-J(s,x)|}{|t-s|^{\frac{1}{4}}}\right|^pds dt.
\]
Then, by Lemma \ref{lem-3}, for any $s,t\in [0,T]$, we have for any $x\in [0,1]$,
\begin{eqnarray*}
|J(t,x)-J(s,x)|&\leq& 8\int^{|t-s|}_0(\rho(x) r^{-2})^{\frac{1}{p}}dr^{\frac{1}{4}}\\
&=& 2\rho^{\frac{1}{p}}(x)\int^{|t-s|}_0 r^{-\frac{2}{p}-\frac{3}{4}}dr.
\end{eqnarray*}
As $p>8$, we have $-\frac{2}{p}-\frac{3}{4}>-1$, which yields
\begin{eqnarray}\label{eq-15}
|J(t,x)-J(s,x)|\leq  C\rho^{\frac{1}{p}}(x)|t-s|^{-\frac{2}{p}+\frac{1}{4}}.
\end{eqnarray}
Taking $s=0$ in (\ref{eq-15}), we have
\begin{eqnarray}\label{eq-16}
|J(t,x)|\leq  C\rho^{\frac{1}{p}}(x)|t|^{-\frac{2}{p}+\frac{1}{4}}\leq C(T,p)\rho^{\frac{1}{p}}(x).
\end{eqnarray}
By utilizing (\ref{eq-31}), (\ref{eq-17}) and (\ref{eq-16}), we know that
\[\int^1_0|K^\varepsilon_3(t,x)|^pdx\le\varepsilon^{p/2}C(T,p)\int^1_0\rho(x)dx,\quad \]
and $E\rho(x)\leq K^pT^2$.

Combining all the previous estimates, we get
\begin{eqnarray*}
&&\int^1_0|u^{\varepsilon}(t\wedge \tau,x)-u^0(t\wedge \tau,x)|^pdx\\
&\leq& L^p(1+R^2+C_0)^{\frac{p}{2}}C(p,T)\int^{t\wedge \tau}_0\int^1_0|u^{\varepsilon}(s\wedge \tau,y)-u^0(s\wedge \tau,y)|^pdyds\\
&& +L^p(1+R^2+C_0)^{\frac{p}{2}}C(p,T)
\int^{t\wedge \tau}_0\int^1_0|u^{\varepsilon}(s\wedge \tau,y)-u^0(s\wedge \tau,y)|^p dyds\\
&& +\varepsilon^{p/2}C(T,p)\int^1_0\rho(x)dx\\.
\end{eqnarray*}
By using Gronwall inequality, we get
\begin{eqnarray}\notag
&&\int^1_0|u^{\varepsilon}(t\wedge \tau,x)-u^0(t\wedge \tau,x)|^pdx\\
\label{eq-33}
&\leq& \Big[\varepsilon^{p/2}C(T,p)\int^1_0\rho(x)dx\Big] \exp\Big\{C(R,p,T,L,C_0)\Big\}.
\end{eqnarray}
Thus,
\begin{eqnarray*}
E\int^1_0|u^{\varepsilon}(t\wedge \tau,x)-u^0(t\wedge \tau,x)|^pdx&\leq& \Big[\varepsilon^{\frac{p}{2}}C(T,p)\int^1_0E\rho(x)dx\Big] \exp\Big\{C(R,p,T,L,C_0)\Big\}\\
&\leq& \varepsilon^{\frac{p}{2}}C(T,p)K^p T^2 \exp\Big\{C(R,p,T,L,C_0)\Big\},
\end{eqnarray*}
which implies (\ref{eqq-1}).
\end{proof}
As a consequence, we have
\begin{cor}\label{cor-1}
For any $p>8$, it holds that
\begin{eqnarray}\label{eq-36}
\sup_{0\leq t\leq T}\sup_{0<\varepsilon\leq 1}E\int^1_0|Y^{\varepsilon}(t\wedge  \tau,x)|^pdx
\leq C_1.
\end{eqnarray}
\end{cor}
Define
$Z^{\varepsilon}=Y^{\varepsilon}-Y=\frac{u^{\varepsilon}-u^0}{\sqrt{\varepsilon}}-Y$, we claim that
\begin{lemma}\label{lem-4}
For any $R>0, p>14$, there exists a constant $C_2=C(K,p,T, L, C_0,C_1)$ such that
\begin{eqnarray*}
\sup_{0\leq t\leq  T}E\int^1_0|Z^{\varepsilon}(t\wedge\tau ,x)|^pdx
\leq \varepsilon^{\frac{p}{2}} C_2.
\end{eqnarray*}
\end{lemma}
\begin{proof}
From (\ref{eqq-11}) and (\ref{eqq-12}), we deduce that
\begin{eqnarray}\notag
&&Z^{\varepsilon}(t\wedge \tau,x)\\ \notag
&=&\int^t_0\int^1_0 G_{t\wedge \tau-s}(x,y)\Big(\frac{b(s,y,u^{\varepsilon}(s,y))-b(s,y,u^{0}(s,y))}{\sqrt{\varepsilon}}-\partial_rb(s,y,u^0(s,y))Y\Big)I_{\{s\leq \tau\}}dyds\\
\notag
&&\  -\int^t_0\int^1_0 \partial_y G_{t\wedge \tau-s}(x,y)\Big(\frac{g(s,y,u^{\varepsilon}(s,y))-g(s,y,u^{0}(s,y))}{\sqrt{\varepsilon}}-\partial_rg(s,y,u^0(s,y))Y\Big)I_{\{s\leq \tau\}}dyds\\
\notag
&&\ +\int^t_0\int^1_0G_{t\wedge \tau-s}(x,y)(\sigma(s,y,u^{\varepsilon}(s,y))-\sigma(s,y,u^{0}(s,y)))I_{\{s\leq \tau\}}W(dyds)\\
\label{eqq-15}
&=:& I^{\varepsilon}_1(t,x)+I^{\varepsilon}_2(t,x)+I^{\varepsilon}_3(t,x) .
\end{eqnarray}
%
With the help of (H3), for $\theta\in (u^{0}(s,y),u^{\varepsilon}(s,y))$, we get
\begin{eqnarray*}
&&\Big|\frac{g(s,y,u^{\varepsilon}(s,y))-g(s,y,u^{0}(s,y))}{\sqrt{\varepsilon}}-\partial_rg(s,y,u^0(s,y))Y\Big|\\
&=&|\partial_rg(s,y,u^0(s,y))Z^{\varepsilon}+\frac{1}{2}\sqrt{\varepsilon}\partial^2_rg(s,y,\theta)|Y^{\varepsilon}|^2|\\
&\leq& K(1+|u^0|)|Z^{\varepsilon}|+\frac{1}{2}\sqrt{\varepsilon}K|Y^{\varepsilon}|^2,
\end{eqnarray*}
then, it yields that
\begin{eqnarray*}
|I^{\varepsilon}_2(t,x)|^p
&\leq& C(p)K^p\Big|\int^t_0\int^1_0 \partial_y G_{t\wedge \tau-s}(x,y)(1+|u^0|)|Z^{\varepsilon}|I_{\{s\leq \tau\}}dyds\Big|^p\\
&& +C(p)\varepsilon^{\frac{p}{2}}K^p\Big|\int^t_0\int^1_0 \partial_y G_{t\wedge \tau-s}(x,y)|Y^{\varepsilon}(s,y)|^2I_{\{s\leq \tau\}}dyds\Big|^p\\
&:=& C(p)K^p(I^{\varepsilon}_{2,1}+\varepsilon^{\frac{p}{2}}I^{\varepsilon}_{2,2}).
\end{eqnarray*}
By H\"{o}lder inequality and (\ref{eqq-4}), for $0<\delta<1$, we get
\begin{eqnarray*}
I^{\varepsilon}_{2,1}&\leq& \Big|\int^{t}_0\Big(\int^1_0 |\partial_y G_{t\wedge \tau-s}(x,y)|^{2\delta}(1+|u^0|^2)dy\Big)^{\frac{1}{2}}\Big(\int^1_0 |\partial_y G_{t\wedge \tau-s}(x,y)|^{2(1-\delta)}|Z^{\varepsilon}(s\wedge \tau)|^2dy\Big)^{\frac{1}{2}}I_{\{s\leq \tau\}}ds\Big|^p\\
&\leq&\Big|\int^{t\wedge \tau}_0(t\wedge \tau-s)^{-\delta}(1+\|u^0(s)\|_H)\Big(\int^1_0 |\partial_y G_{t\wedge \tau-s}(x,y)|^{2(1-\delta)}|Z^{\varepsilon}(s\wedge \tau,y)|^2dy\Big)^{\frac{1}{2}}ds\Big|^p\\
&\leq&(1+C_0)^{\frac{p}{2}}\Big|\int^{t\wedge \tau}_0(t\wedge \tau-s)^{-\delta}\Big(\int^1_0 |\partial_y G_{t\wedge \tau-s}(x,y)|^{2(1-\delta)}|Z^{\varepsilon}(s\wedge \tau,y)|^2dy\Big)^{\frac{1}{2}}ds\Big|^p\\
&\leq&(1+C_0)^{\frac{p}{2}}\big(\int^{t\wedge \tau}_0(t\wedge \tau-s)^{-2\delta}ds\big)^{\frac{p}{2}}\Big(\int^{t\wedge \tau}_0\int^1_0|\partial_y G_{t\wedge \tau-s}(x,y)|^{2(1-\delta)}|Z^{\varepsilon}(s\wedge \tau,y)|^2dyds\Big)^{\frac{p}{2}}\\
&\leq&(1+C_0)^{\frac{p}{2}}\Big(\int^{t\wedge \tau}_0(t\wedge \tau-s)^{-2\delta}ds\Big)^{\frac{p}{2}}\Big(\int^{t\wedge \tau}_0\int^1_0|\partial_y G_{t\wedge \tau-s}(x,y)|^{2(1-\delta) q}dyds\Big)^{\frac{p}{2q}}\\
&& \times\Big(\int^{t\wedge \tau}_0\int^1_0|Z^{\varepsilon}(s\wedge \tau,y)|^pdyds\Big),
\end{eqnarray*}
where $\frac{2}{p}+\frac{1}{q}=1$.

As $p>8$, we have $1<q<\frac{4}{3}$, taking $\delta=\frac{15}{32}$, it yields
\[
-2 \delta>-1, \quad 0<2(1-\delta)q<\frac{3}{2}.
\]
Then, by (\ref{eqq-6}), we get
\begin{eqnarray}\notag
I^{\varepsilon}_{2,1}
\label{eee-9}
\leq (1+C_0)^{\frac{p}{2}}C(T,p)\int^{t\wedge \tau}_0\int^1_0|Z^{\varepsilon}(s\wedge \tau,y)|^pdyds.
\end{eqnarray}
By H\"{o}lder inequality, we deduce that
\begin{eqnarray*}
I^{\varepsilon}_{2,2}
\leq \Big(\int^{t\wedge \tau}_0\int^1_0 |\partial_y G_{t\wedge \tau-s}(x,y)|^rdyds\Big)^{\frac{p}{r}}\Big(\int^t_0\int^1_0|Y^{\varepsilon}(s\wedge \tau,y)|^{2p}dyds\Big),
\end{eqnarray*}
where $\frac{1}{r}+\frac{1}{p}=1$. As $p>8$, we have $1<r<\frac{8}{7}<\frac{3}{2}$, by (\ref{eqq-6}), we get
\begin{eqnarray}\notag
I^{\varepsilon}_{2,2}\label{eq-34}
\leq C(T, p)\Big(\int^t_0\int^1_0|Y^{\varepsilon}(s\wedge \tau,y)|^{2p}dyds\Big).
\end{eqnarray}

Combining (\ref{eee-9}) and (\ref{eq-34}), we deduce that
\begin{eqnarray*}
|I^{\varepsilon}_2(t,x)|^p&\leq & K^p(1+C_0)^{\frac{p}{2}}C(T,p)\int^{t\wedge \tau}_0\int^1_0|Z^{\varepsilon}(s\wedge \tau,y)|^pdyds\\
&& +C(K,T,R,p) \varepsilon^{\frac{p}{2}}\left(\int^{t}_0\int^1_0 |Y^{\varepsilon}(s\wedge \tau,y)|^{2p}dyds\right).
\end{eqnarray*}
Similar to the proof of $I^{\varepsilon}_2(t,x)$, we get
\begin{eqnarray*}
|I^{\varepsilon}_1(t,x)|^p
&\leq& K^p(1+C_0)^{\frac{p}{2}}C(T,p)\int^{t\wedge \tau}_0\int^1_0|Z^{\varepsilon}(y,s\wedge \tau)|^pdyds
\\
&& +C(K,T,R,p) \varepsilon^{\frac{p}{2}}\left(\int^{t}_0\int^1_0 |Y^{\varepsilon}(s\wedge \tau,y)|^{2p}dyds\right).
\end{eqnarray*}
To estimate $I^{\varepsilon}_3$, we define
\[
J(t,x)=\int^t_0\int^1_0G_{t-s}(x,y)(\sigma(s,y,u^{\varepsilon}(s,y))-\sigma(s,y,u^{0}(s,y)))I_{\{s\leq \tau\}}W(dyds).
\]
Then,
\begin{eqnarray}\label{eq-37}
I^{\varepsilon}_3(t,x)=J(t\wedge\tau,x).
\end{eqnarray}
Note that for any $0\le s<t\le T, x\in [0,1]$, by Burkholder-Davis-Gundy inequality, (H1) and (\ref{eqq-5-1}), for some $\kappa\geq 1$, we obtain
\begin{eqnarray}\notag
&&E|J(t,x)-J(s,x)|^p\\ \notag
&\le&E\left|\int^s_0\int^1_0(G_{t-r}(x,y)-G_{s-r}(x,y))(\sigma(r,y,u^{\varepsilon}(r,y))-\sigma(r,y,u^{0}(r,y)))I_{\{r\leq \tau\}}W(dydr)\right|^p\\ \notag
&&+E\left|\int^t_s\int^1_0G_{t-r}(x,y)(\sigma(r,y,u^{\varepsilon}(r,y))-\sigma(r,y,u^{0}(r,y)))I_{\{r\leq \tau\}}W(dydr)\right|^p\\ \notag
&\le&L^pE\left|\int^s_0\int^1_0(G_{t-r}(x,y)-G_{s-r}(x,y))^2|u^{\varepsilon}(r,y)-u^{0}(r,y)|^2I_{\{r\leq \tau\}}dydr\right|^{p/2}\\ \notag
&&+L^pE\left|\int^t_s\int^1_0G_{t-r}(x,y)^2|u^{\varepsilon}(r,y)-u^{0}(r,y)|^2I_{\{r\leq \tau\}}dydr\right|^{p/2}\\ \notag
&\leq& L^p\Big(\int^s_0\int^1_0|G_{t-r}(x,z)-G_{s-r}(x,y)|^{2q'}dy dr\Big)^{\frac{p}{2q'}}\times E\Big(\int^t_0\int^1_0|u^{\varepsilon}(r\wedge \tau,y)-u^{0}(r\wedge \tau,y)|^{2p'}dydr\Big)^{\frac{p}{2p'}}\\ \notag
&& +L^p\Big(\int^t_s\int^1_0G_{t-r}(x,y)^{2q'}dydr\Big)^{\frac{p}{2q'}}\times E\Big(\int^t_s\int^1_0|u^{\varepsilon}(r\wedge \tau,y)-u^{0}(r\wedge \tau,y)|^{2p'}dydr\Big)^{\frac{p}{2p'}}\\ \notag
&\leq& L^p\Big(\int^s_0\int^1_0|G_{t-r}(x,z)-G_{s-r}(x,y)|^{2q'}dy dr\Big)^{\frac{p}{2q'}}\times  E\Big(\int^t_0\int^1_0|u^{\varepsilon}(r\wedge \tau,y)-u^{0}(r\wedge \tau,y)|^{2p'\kappa}dydr\Big)^{\frac{p}{2p'\kappa}}\\
\label{eq-17-1}
&& +L^p\Big(\int^t_s\int^1_0G_{t-r}(x,y)^{2q'}dydr\Big)^{\frac{p}{2q'}}\times E\Big(\int^t_s\int^1_0|u^{\varepsilon}(r\wedge \tau,y)-u^{0}(r\wedge \tau,y)|^{2p'\kappa}dydr\Big)^{\frac{p}{2p'\kappa}}.
\end{eqnarray}
where $\frac{1}{p'}+\frac{1}{q'}=1$.

Taking $\kappa=\frac{p}{2p'}$. When $p>14$, we have $\frac{p}{p-2}<\frac{3p}{8+2p}<\frac{3}{2}$, choosing $\frac{p}{p-2}<q'<\frac{3p}{8+2p}$, then
\[
2q'<3, \quad \kappa>1,\quad \frac{(3-2q')p}{4q'}>2.
\]
As a result, by Lemma \ref{lem-2}, we deduce that
\begin{eqnarray}\notag
E|J(t,x)-J(s,x)|^p
&\leq& C(T)L^p|t-s|^{\frac{(3-2q')p}{4q'}} \Big(\int^t_0E\int^1_0|u^{\varepsilon}(r\wedge \tau,y)-u^{0}(r\wedge \tau,y)|^{p}dydr\Big)\\ \notag
&& +C(T)L^p|t-s|^{\frac{(3-2q')p}{4q'}} \Big(\int^t_sE\int^1_0|u^{\varepsilon}(r\wedge \tau,y)-u^{0}(r\wedge \tau,y)|^{p}dydr\Big)\\
\label{eq-38}
&\leq&  \varepsilon^{\frac{p}{2}}C(L,p,T,C_1)|t-s|^{\frac{(3-2q')p}{4q'}}.
\end{eqnarray}
Let
\[
\Psi(r)=r^p,\quad p(r)=r^{\frac{(3-2q')}{4q'}},
\]
and
\[
\rho(x)=\int^T_0\int^T_0\Big|\frac{|J(t,x)-J(s,x)|}{|t-s|^{\frac{(3-2q')}{4q'}}}\Big|^pdsdt.
\]
Then, by Lemma \ref{lem-3}, for any $s,t\in [0,T]$, $x\in [0,1]$, we have
\begin{eqnarray}\notag
|J(t,x)-J(s,x)|&\leq& 8\int^{|t-s|}_0(\rho(x)r^{-2})^{\frac{1}{p}}dr^{\frac{(3-2q')}{4q'}}\\ \notag
&\leq& C\rho^{\frac{1}{p}}(x)\int^{|t-s|}_0r^{-\frac{2}{p}+\frac{(3-2q')}{4q'}-1}dr.
\end{eqnarray}
As $\frac{(3-2q')p}{4q'}>2$, then $-\frac{2}{p}+\frac{(3-2q')}{4q'}>0$, we get
\begin{eqnarray}\label{eq-23-1}
|J(t,x)-J(s,x)|\leq C\rho^{\frac{1}{p}}(x)|t-s|^{-\frac{2}{p}+\frac{(3-2q')}{4q'}}.
\end{eqnarray}
Taking $s=0$ in (\ref{eq-23-1}), we obtain
\begin{eqnarray}\notag
|J(t,x)|\leq C(T)\rho^{\frac{1}{p}}(x).
\end{eqnarray}
Utilizing (\ref{eq-37}) and (\ref{eq-38}), we deduce that
\begin{eqnarray}\label{eee-1}
\int^1_0|I^{\varepsilon}_3(t,x)|^pdx\le C(T)\int^1_0\rho(x) dx
\end{eqnarray}
and
\begin{equation}\label{eq-26}
 E\int^1_0\rho(x) dx\le \varepsilon^{\frac{p}{2}}C(L,p,T,C_1)T^2.
\end{equation}
Combining all the above estimates,  we get
\begin{eqnarray*}
\int^1_0|Z^{\varepsilon}(t\wedge \tau,x)|^pdx&\leq& K^p(1+C_0)^{\frac{p}{2}}C(T,p)\int^{t\wedge \tau}_0\int^1_0|Z^{\varepsilon}(y,s\wedge \tau)|^pdyds
\\
&& +K^p(1+C_0)^{\frac{p}{2}}C(T,p)\int^{t\wedge \tau}_0\int^1_0|Z^{\varepsilon}(s\wedge \tau,y)|^pdyds\\
&& +C(K,T,R,p) \varepsilon^{\frac{p}{2}}\left(\int^T_0\int^1_0 |Y^{\varepsilon}(s\wedge \tau,y)|^{2p}dyds\right)+C(T)\int^1_0\rho(x) dx.
\end{eqnarray*}
By Gronwall inequality, it follows that
\begin{eqnarray*}
&&\int^1_0|Z^{\varepsilon}(t\wedge \tau,x)|^pdx\\
&\leq& \Big[C(K,T,R,p,C_0) \varepsilon^{\frac{p}{2}}\left(\int^{T}_0\int^1_0 |Y^{\varepsilon}(s\wedge \tau,y)|^{2p}dyds\right) +C(T)\int^1_0\rho(x)dx\Big]\\
&& \times \exp\Big\{K^p(1+C_0)^{\frac{p}{2}}C(T,p)\Big\}.
\end{eqnarray*}
Taking expectation, by (\ref{eq-36}) and (\ref{eq-26}), we get
\begin{eqnarray*}
&&E\int^1_0|Z^{\varepsilon}(t\wedge \tau,x)|^pdx\\
&\leq& \Big[C(K,T,R,p,C_0) \varepsilon^{\frac{p}{2}}C_1T +\varepsilon^{\frac{p}{2}}C(L,p,T,C_1)\Big] \exp\{K^p(1+C_0)^{\frac{p}{2}}C(T,p)\}.
\end{eqnarray*}
We complete the proof.

\end{proof}

\subsection{Proof of CLT for semilinear SPDE}
\textbf{Proof of Theorem \ref{thm-2}}. \quad
Recall $\tau^{\varepsilon,R}$ is defined by (\ref{eee-30}). For any $\delta>0$, it follows that
\begin{eqnarray*}
&&P\Big(\sup_{t\in [0,T]}\|Y^{\varepsilon}(t)-Y(t)\|_H>\delta\Big)\\
&\leq & P\Big(\sup_{t\in [0,T]}\|Y^{\varepsilon}(t)-Y(t)\|_H>\delta, \tau^{\varepsilon, R}\leq T\Big)+P\Big(\sup_{t\in [0,T]}\|Y^{\varepsilon}(t)-Y(t)\|_H>\delta, \tau^{\varepsilon, R}> T\Big)\\
&\leq & P\Big(\tau^{\varepsilon, R}\leq T\Big)+P\Big(\sup_{t\in [0,\tau^{\varepsilon, R}]}\|Y^{\varepsilon}(t)-Y(t)\|_H>\delta\Big).
\end{eqnarray*}
By (\ref{eqq-13}), we get for any $\varepsilon\in (0,1]$,
\begin{eqnarray*}
P(\tau^{\varepsilon, R}\leq T)\rightarrow 0,\quad {\rm{as}} \ R\rightarrow \infty.
\end{eqnarray*}
Fix some $R>0$, denote by $\tau=\tau^{\varepsilon,R}$.
Recall $Z^{\varepsilon}(t\wedge \tau,x)=Y^{\varepsilon}(t\wedge \tau,x)-Y=\frac{u^{\varepsilon}(t\wedge \tau,x)-u^{0}(t\wedge \tau,x)}{\sqrt{\varepsilon}}-Y$ satisfies (\ref{eqq-15}). For the readers' convenience, we state it again as follows.
\begin{eqnarray}\notag
Z^{\varepsilon}(t\wedge \tau,x)&=&\int^t_0\int^1_0G_{t\wedge \tau-s}(x,y)\Big(\frac{b(u^{\varepsilon})-b(u^0)}{\sqrt{\varepsilon}}-\partial_rb(s,y,u^0(s,y))Y\Big)I_{\{s\leq \tau\}}dsdy\\
\notag
&&\ -\int^t_0\int^1_0\partial_yG_{t\wedge \tau-s}(x,y)\Big(\frac{g(u^{\varepsilon})-g(u^0)}{\sqrt{\varepsilon}}-\partial_rg(s,y,u^0(s,y))Y\Big)I_{\{s\leq \tau\}}dsdy\\
\notag
&&\ +\int^t_0\int^1_0G_{t\wedge \tau-s}(x,y)\Big(\sigma(s,y,u^{\varepsilon}(s,y))-\sigma(s,y,u^0(s,y))\Big)I_{\{s\leq \tau\}}W(dyds)\\
\label{eee-21}
&:=& I^{\varepsilon}_1(t,x)+I^{\varepsilon}_2(t,x)+I^{\varepsilon}_3(t,x).
\end{eqnarray}
In the rest part, we aim to prove $Z^{\varepsilon}(t\wedge \tau,x)\rightarrow 0$ in probability in $C([0, T]; H)$ as $\varepsilon\rightarrow 0$.

By (\ref{eee-1}) and (\ref{eq-26}), for $p>14$, it yields
\begin{eqnarray}\notag
E\sup_{t\in [0,T]}\int^1_0|I^{\varepsilon}_3(t,x)|^pdx\leq  \varepsilon^{\frac{p}{2}}C(L,p,T,C_1).
\end{eqnarray}
By Chebyshev inequality, we get for the above $\delta>0$,
\begin{eqnarray}\notag
P\Big(\sup_{t\in [0,T]}\|I^{\varepsilon}_3(t)\|_H>\delta\Big)
&\leq& \frac{ E\sup_{t\in [0,T]}\|I^{\varepsilon}_3(t)\|^p_H}{\delta^p}\\ \notag
&\leq&\frac{C E\sup_{t\in [0,T]}\|I^{\varepsilon}_3(t)\|^p_p}{\delta^p}\\ \notag
&\leq& \frac{ C E\sup_{t\in [0,T]}\int^1_0|I^{\varepsilon}_3(t,x)|^pdx}{\delta^p}\\ \notag
&\leq& \varepsilon^{\frac{p}{2}}\frac{ C(L,p,T,C_1)}{\delta^p}\\
\label{eq-27}
& \rightarrow& 0,\quad {\rm{as}}\ \varepsilon\rightarrow 0,
\end{eqnarray}
i.e. $I^{\varepsilon}_3(t,x)\rightarrow 0$ in probability in $C([0, T]; H)$ as $\varepsilon\rightarrow 0$.

Define
\begin{eqnarray*}
\bar{I}^{\varepsilon}_2(t,x):=-\int^t_0\int^1_0\partial_yG_{t-s}(x,y)\Big(\frac{g(u^{\varepsilon})-g(u^0)}{\sqrt{\varepsilon}}-\partial_rg(s,y,u^0(s,y))Y\Big)I_{\{s\leq \tau\}}dsdy,
\end{eqnarray*}
then
\begin{eqnarray}\label{eq-30}
{I}^{\varepsilon}_2(t,x)=\bar{I}^{\varepsilon}_2(t\wedge \tau,x).
\end{eqnarray}
Note that for $t_1,t_2\in [0,T], t_1>t_2$, we have
\begin{eqnarray*}
&&\bar{I}^{\varepsilon}_2(t_1,x)-\bar{I}^{\varepsilon}_2(t_2,x)\\
&=&\int^{t_1}_0\int^1_0\partial_yG_{t_1-s}(x,y)\Big[\frac{g(u^{\varepsilon})-g(u^0)}{\sqrt{\varepsilon}}-\partial_rg(s,y,u^0(s,y))Y\Big]I_{\{s\leq \tau\}}dsdy\\
&& -\int^{t_2}_0\int^1_0\partial_yG_{t_2-s}(x,y)\Big[\frac{g(u^{\varepsilon})-g(u^0)}{\sqrt{\varepsilon}}-\partial_rg(s,y,u^0(s,y))Y\Big]I_{\{s\leq \tau\}}dsdy\\
&=& \int^{t_1}_{t_2}\int^1_0\partial_yG_{t_1-s}(x,y)\Big[\frac{g(u^{\varepsilon})-g(u^0)}{\sqrt{\varepsilon}}-\partial_rg(s,y,u^0(s,y))Y\Big]I_{\{s\leq \tau\}}dsdy\\
&& +\int^{t_2}_0\int^1_0\partial_y(G_{t_1-s}(x,y)-G_{t_2-s}(x,y))\Big[\frac{g(u^{\varepsilon})-g(u^0)}{\sqrt{\varepsilon}}-\partial_rg(s,y,u^0(s,y))Y\Big]I_{\{s\leq \tau\}}dsdy.
\end{eqnarray*}
By (H3), for $\theta\in (u^0(y,s), u^{\varepsilon}(y,s))$, we get
\begin{eqnarray*}
&&\frac{g(u^{\varepsilon}(y,s))-g(u^0(y,s))}{\sqrt{\varepsilon}}-\partial_rg(s,y,u^0(s,y))Y\\
&=& \frac{\partial_rg(s,y,u^0(s,y))(u^{\varepsilon}-u^0)+\frac{1}{2}\partial^2_rg(\theta)(u^{\varepsilon}-u^0)^2}{\sqrt{\varepsilon}}-\partial_rg(s,y,u^0(s,y))Y\\
&=& \partial_rg(s,y,u^0(s,y))Z^{\varepsilon}+\frac{1}{2}\sqrt{\varepsilon}\partial^2_rg(\theta)|Y^{\varepsilon}|^2\\
&\leq& K(1+|u^0|)|Z^{\varepsilon}|+\frac{1}{2}K\sqrt{\varepsilon}|Y^{\varepsilon}|^2.
\end{eqnarray*}
Then, it follows that
\begin{eqnarray}\notag
&&|\bar{I}^{\varepsilon}_2(t_1,x)-\bar{I}^{\varepsilon}_2(t_2,x)|\\ \notag
&\leq& K\sqrt{\varepsilon}\int^{t_1}_{t_2}\int^1_0\partial_yG_{t_1-s}(x,y)(1+|u^0|)\Big|\frac{Z^{\varepsilon}(s)}{\sqrt{\varepsilon}}\Big|I_{\{s\leq \tau\}}dsdy\\ \notag
&& +\frac{1}{2}K\sqrt{\varepsilon}\int^{t_1}_{t_2}\int^1_0\partial_yG_{t_1-s}(x,y)|Y^{\varepsilon}|^2I_{\{s\leq \tau\}}dsdy\\ \notag
&& +K\sqrt{\varepsilon}\int^{t_2}_0\int^1_0\partial_y(G_{t_1-s}(x,y)-G_{t_2-s}(x,y))(1+|u^0|)\Big|\frac{Z^{\varepsilon}(s)}{\sqrt{\varepsilon}}\Big|I_{\{s\leq \tau\}}dsdy\\ \notag
&& +\frac{1}{2}K\sqrt{\varepsilon}\int^{t_2}_0\int^1_0\partial_y(G_{t_1-s}(x,y)-G_{t_2-s}(x,y))|Y^{\varepsilon}|^2I_{\{s\leq \tau\}}dsdy\\
\label{eq-13}
&:=& \sqrt{\varepsilon}(\bar{I}^{\varepsilon}_{2,1}+\bar{I}^{\varepsilon}_{2,2}+\bar{I}^{\varepsilon}_{2,3}+\bar{I}^{\varepsilon}_{2,4}).
\end{eqnarray}
In the rest part, we take $p>14$.
By H\"{o}lder inequality and Lemma \ref{lem-4}, for some $0<\delta_1<1$, we deduce that
\begin{eqnarray*}
 &&E\|\bar{I}^{\varepsilon}_{2,1}\|^p_H\\
 &=& E\Big[\int^1_0|\int^{t_1}_{t_2}\int^1_0\partial_yG_{t_1-s}(x,y)(1+|u^0|)\Big|\frac{Z^{\varepsilon}(s)}{\sqrt{\varepsilon}}\Big|I_{\{s\leq \tau\}}dyds|^2dx\Big]^{\frac{p}{2}}\\
 &\leq& E\Big[\int^1_0\Big(\int^{t_1}_{t_2}\int^1_0|\partial_yG_{t_1-s}(x,y)|^{2\delta_1}(1+|u^0|)^2 dyds\Big)\Big(\int^{t_1}_{t_2}\int^1_0|\partial_yG_{t_1-s}(x,y)|^{2(1-\delta_1)}\Big|\frac{Z^{\varepsilon}(s)}{\sqrt{\varepsilon}}\Big|^2I_{\{s\leq \tau\}}dyds\Big)dx\Big]^{\frac{p}{2}}\\
 &\leq& \Big(\int^{t_1}_{t_2}(t_1-s)^{-2\delta_1}(1+\|u^0\|^2_H)ds\Big)^{\frac{p}{2}}E\Big[\int^1_0\Big(\int^{t_1\wedge \tau}_{t_2}\int^1_0|\partial_yG_{t_1-s}(x,y)|^{2(1-\delta_1) q}dyds\Big)^{\frac{p}{2q}}\\
&& \times \Big(\int^1_0\int^{t_1}_{t_2}\int^1_0\Big|\frac{Z^{\varepsilon}(s\wedge \tau)}{\sqrt{\varepsilon}}\Big|^pdydsdx\Big)dx\Big]^{\frac{p}{2}}\\
&\leq& K^p(1+C_0)^{\frac{p}{2}}\Big(\int^{t_1}_{t_2}(t_1-s)^{-2\delta_1}ds\Big)^{\frac{p}{2}}E\Big[\int^1_0\Big(\int^{t_1 }_{t_2}\int^1_0|\partial_yG_{t_1-s}(x,y)|^{2(1-\delta_1) q}dyds\Big)^{\frac{p}{2q}}\\
&& \times \Big(\int^1_0\int^{t_1}_{t_2}\int^1_0\Big|\frac{Z^{\varepsilon}(s\wedge \tau)}{\sqrt{\varepsilon}}\Big|^pdydsdx\Big)^{\frac{2}{p}}dx\Big]^{\frac{p}{2}}.
\end{eqnarray*}
where $\frac{2}{p}+\frac{1}{q}=1$.

Taking $\delta_1\in ( \frac{5}{14}, \frac{1}{2})$, as $p>14$, we have
\[
-2\delta_1>-1,\ 2(1-\delta_1) q<\frac{3}{2}.
\]
Set
\begin{eqnarray*}
  b_1=\frac{1}{2}-2(1-\delta_1) q,
\end{eqnarray*}
then by Lemma \ref{lem-4}, we deduce that
\begin{eqnarray*}
 E\|\bar{I}^{\varepsilon}_{2,1}\|^p_H&\leq& K^p(1+C_0)^{\frac{p}{2}}\Big(\int^{t_1}_{t_2}(t_1-s)^{-2\delta_1}ds\Big)^{\frac{p}{2}}(t_1-t_2)^{(b_1+1)\frac{p}{2q}}\\
 && \times \int^{t_1}_{t_2}E\int^1_0\Big|\frac{Z^{\varepsilon}(s\wedge \tau)}{\sqrt{\varepsilon}}\Big|^pdyds\\
&\leq&C_2K^p(1+C_0)^{\frac{p}{2}}(t_1-t_2)^{\frac{(-2\delta_1+1)p}{2}}(t_1-t_2)^{(b_1+1)\frac{p}{2q}}(t_1-t_2)\\
&\leq&C_2K^p(1+C_0)^{\frac{p}{2}}|t_1-t_2|^{\alpha_1},
\end{eqnarray*}
where
\[
\alpha_1=\frac{(-2\delta_1+1)p}{2}+(b_1+1)\frac{p}{2q}+1=\frac{p-2}{4}.
\]
Thus,
\begin{eqnarray}\label{eee-14}
E\|\bar{I}^{\varepsilon}_{2,1}\|^p_H\leq C(K,p, C_0,C_2)|t_1-t_2|^{\frac{p-2}{4}}.
\end{eqnarray}

Utilizing H\"{o}lder inequality and Corollary \ref{cor-1}, we get
\begin{eqnarray*}
 E\|\bar{I}^{\varepsilon}_{2,2}\|^p_H&\leq&
CK^pE\Big[\int^1_0|\int^{t_1}_{t_2}\int^1_0\partial_yG_{t_1-s}(x,y)|Y^{\varepsilon}(s\wedge \tau)|^2dsdy|^2dx\Big]^{\frac{p}{2}}\\
&\leq& CK^p E\Big[\int^1_0\Big(\int^{t_1}_{t_2}\int^1_0|\partial_yG_{t_1-s}(x,y)|^{r}dyds\Big)^{\frac{2}{r}}\Big(\int^{t_1}_{t_2}\int^1_0|Y^{\varepsilon}(s\wedge \tau)|^pdyds\Big)^{\frac{4}{p}}dx\Big]^{\frac{p}{2}}\\
&\leq& CK^p |t_1-t_2|^{\frac{(3-2r)p}{2r} } E\Big[\int^1_0\Big(\int^{t_1}_{t_2}\int^1_0|Y^{\varepsilon}(s\wedge \tau, y)|^{2p}dyds\Big)^{\frac{2}{p}}\Big(\int^{t_1}_{t_2}\int^1_0dyds\Big)^{\frac{2}{p}}dx\Big]^{\frac{p}{2}}\\
&\leq& CK^p |t_1-t_2|^{\frac{(3-2r)p}{2r}+1 } E\Big(\int^{t_1}_{t_2}\int^1_0|Y^{\varepsilon}(s\wedge \tau, y)|^{2p}dyds\Big),
\end{eqnarray*}
 where $\frac{2}{p}+\frac{1}{r}=1$.

As $p>14$, by (\ref{eq-36}), we get
\begin{eqnarray}\label{eee-15}
 E\|\bar{I}^{\varepsilon}_{2,2}\|^p_H\leq  CK^p|t_1-t_2|^{\frac{(3-2r)p}{2r}+2}C_1= C(p,K,C_1)|t_1-t_2|^{\frac{p-2}{2}}.
\end{eqnarray}
By the definition of heat kernel $G$, for $t_1>t_2>s$, we deduce that
\begin{eqnarray*}
&&\partial_y(G_{t_1-s}(x,y)-G_{t_2-s}(x,y))\\
&=&\frac{1}{\sqrt{2\pi}}\partial_y\Big[\frac{1}{\sqrt{t_1-s}}e^{-\frac{(x-y)^2}{2(t_1-s)}}-\frac{1}{\sqrt{t_2-s}}e^{-\frac{(x-y)^2}{2(t_2-s)}}\Big]\\
&=&\frac{1}{\sqrt{2\pi}}\Big[\frac{1}{\sqrt{t_1-s}}\frac{(x-y)}{t_1-s}e^{-\frac{(x-y)^2}{2(t_1-s)}}-\frac{1}{\sqrt{t_2-s}}\frac{(x-y)}{t_2-s}e^{-\frac{(x-y)^2}{2(t_2-s)}}\Big]\\
&=& \frac{1}{\sqrt{t_1-s}}\tilde{G}_{t_1-s}(x,y)-\frac{1}{\sqrt{t_2-s}}\tilde{G}_{t_2-s}(x,y)\\
&=& \Big(\frac{1}{\sqrt{t_1-s}}-\frac{1}{\sqrt{t_2-s}}\Big)\tilde{G}_{t_2-s}(x,y)+\frac{1}{\sqrt{t_1-s}}\Big(\tilde{G}_{t_1-s}(x,y)-\tilde{G}_{t_2-s}(x,y)\Big)\\
&\leq& \frac{t_1-t_2}{(t_1-s)\sqrt{t_2-s}}\tilde{G}_{t_2-s}(x,y)+\frac{1}{\sqrt{t_1-s}}\Big(\tilde{G}_{t_1-s}(x,y)-\tilde{G}_{t_2-s}(x,y)\Big).
\end{eqnarray*}
Define
\[
\tilde{G}_{t}(x,y):=\frac{1}{\sqrt{2\pi}}\frac{(x-y)}{t}e^{-\frac{(x-y)^2}{2t}},
\]
with the help of properties of Gamma function, we establish that $\tilde{G}^r_{t}(x,y)$ satisfies (\ref{eqq-4})-(\ref{eq-29}). Then, it follows that
\begin{eqnarray*}
&&E\|\bar{I}^{\varepsilon}_{2,3}\|^p_H\\
&\leq&E\Big[\int^1_0\Big|K\int^{t_2}_0\int^1_0\partial_y(G_{t_1-s}(x,y)-G_{t_2-s}(x,y))(1+|u^0|)|\frac{Z^{\varepsilon}(s)}{\sqrt{\varepsilon}}|I_{\{s\leq \tau\}}dsdy\Big|^2dx\Big]^{\frac{p}{2}}\\
&\leq& C(p) E\Big[\int^1_0\Big|K\int^{t_2}_0\int^1_0\frac{t_1-t_2}{(t_1-s)\sqrt{t_2-s}}\tilde{G}_{t_2-s}(x,y)(1+|u^0|)|\frac{Z^{\varepsilon}(s)}{\sqrt{\varepsilon}}|I_{\{s\leq \tau\}}dsdy\Big|^2dx\Big]^{\frac{p}{2}}\\
&& +C(p)E\Big[\int^1_0\Big|K\int^{t_2}_0\int^1_0\frac{1}{\sqrt{t_1-s}}(\tilde{G}_{t_1-s}(x,y)-\tilde{G}_{t_2-s}(x,y))(1+|u^0|)|\frac{Z^{\varepsilon}(s)}{\sqrt{\varepsilon}}|I_{\{s\leq \tau\}}dsdy\Big|^2dx\Big]^{\frac{p}{2}}\\
&\leq& \varepsilon^{-\frac{p}{2}} C(p) E\Big[\int^1_0\Big|K\int^{t_2}_0\int^1_0\frac{t_1-t_2}{(t_1-s)\sqrt{t_2-s}}\tilde{G}_{t_2-s}(x,y)(1+|u^0|)|Z^{\varepsilon}(s\wedge \tau)|dsdy\Big|^2dx\Big]^{\frac{p}{2}}\\
&& +\varepsilon^{-\frac{p}{2}}C(p)E\Big[\int^1_0\Big|K\int^{t_2}_0\int^1_0\frac{1}{\sqrt{t_1-s}}(\tilde{G}_{t_1-s}(x,y)-\tilde{G}_{t_2-s}(x,y))(1+|u^0|)|Z^{\varepsilon}(s\wedge \tau)|dsdy\Big|^2dx\Big]^{\frac{p}{2}}\\
&=:& \varepsilon^{-\frac{p}{2}}C(p)(K_1+K_2).
\end{eqnarray*}
By (\ref{eqq-5}), H\"{o}lder inequality and Lemma \ref{lem-5}, for $\alpha_0>0$, we have
\begin{eqnarray*}
K_1&\leq& K^p E\Big[\int^1_0\Big|\int^{t_2}_0\frac{t_1-t_2}{(t_1-s)\sqrt{t_2-s}}\Big(\int^1_0\tilde{G}^{2+{\alpha_0}}_{t_2-s}
(x,y)dy\Big)^{\frac{1}{2+{\alpha_0}}}(1+\|u^0\|_H)\Big(\int^1_0|Z^{\varepsilon}(s\wedge \tau)|^{p_1}dy \Big)^{\frac{1}{p_1}} ds\Big|^2dx\Big]^{\frac{p}{2}}\\
&\leq& K^p(1+\sqrt{C_0})^pE\Big[\int^1_0\Big|\int^{t_2}_0\frac{t_1-t_2}{(t_1-s)\sqrt{t_2-s}}(t_2-s)^{-\frac{1+{\alpha_0}}{2(2+{\alpha_0})}}\|Z^{\varepsilon}(s\wedge \tau)\|_{L^{p_1}}ds\Big|^2dx\Big]^{\frac{p}{2}}\\
&=&K^p(1+\sqrt{C_0})^pE\Big|\int^{t_2}_0\frac{t_1-t_2}{(t_1-s)\sqrt{t_2-s}}(t_2-s)^{-\frac{1+{\alpha_0}}{2(2+{\alpha_0})}}\|Z^{\varepsilon}(s\wedge \tau)\|_{L^{p_1}}ds\Big|^p\\
&=& K^p(1+\sqrt{C_0})^p\int^{t_2}_0\cdot\cdot\cdot\int^{t_2}_0\prod^p_{k=1}\frac{t_1-t_2}{(t_1-s_k)\sqrt{t_2-s_k}}(t_2-s_k)^{-\frac{1+{\alpha_0}}{2(2+{\alpha_0})}}
E\prod^p_{k=1}\|Z^{\varepsilon}(s_k\wedge \tau)\|_{L^{p_1}}ds_1\cdot\cdot\cdot ds_p,
\end{eqnarray*}
where
\[
 \frac{1}{2+{\alpha_0}}+\frac{1}{p_1}=\frac{1}{2}.
\]
By Cauchy-Schwarz inequality with $r_1,\cdot\cdot\cdot, r_p$  satisfying $\sum^p_{k=1}\frac{1}{ r_k}=1$, we deduce that
\begin{eqnarray*}
E\prod^p_{k=1}\|Z^{\varepsilon}(s_k\wedge \tau)\|_{L^{p_1}}
\leq (E\|Z^{\varepsilon}(s_1\wedge \tau)\|^{r_1}_{L^{p_1}})^{\frac{1}{ r_1}}\cdot\cdot\cdot(E\|Z^{\varepsilon}(s_p\wedge \tau)\|^{r_p}_{L^{p_1}})^{\frac{1}{ r_p}}.
\end{eqnarray*}
Let $\tilde{r}_k=r_k\vee p_1, k=1,\cdot\cdot\cdot,p$, by H\"{o}lder inequality, we get
\[
(E\|Z^{\varepsilon}(s_k\wedge \tau)\|^{r_k}_{L^{p_1}})^{\frac{1}{ r_k}}\leq (E\|Z^{\varepsilon}(s_k\wedge \tau)\|^{r_k}_{L^{\tilde{r}_k}})^{\frac{1}{ r_k}}\leq (E\|Z^{\varepsilon}(s_k\wedge \tau)\|^{\tilde{r}_k}_{L^{\tilde{r}_k}})^{\frac{1}{ \tilde{r}_k}},
\]
thus,
\begin{eqnarray*}
E\prod^p_{k=1}\|Z^{\varepsilon}(s_k\wedge \tau)\|_{L^{p_1}}
\leq (E\|Z^{\varepsilon}(s_1\wedge \tau)\|^{\tilde{r}_1}_{L^{\tilde{r}_1}})^{\frac{1}{ \tilde{r}_1}}\cdot\cdot\cdot(E\|Z^{\varepsilon}(s_p\wedge \tau)\|^{\tilde{r}_p}_{L^{\tilde{r}_p}})^{\frac{1}{ \tilde{r}_p}}.
\end{eqnarray*}
Choosing $\alpha_0=1$, then $p_1=6$. Taking $r_k= p>14$, for $k=1,2,\cdot\cdot\cdot,p$.
With the aid of Lemma \ref{lem-4} with $\tilde{r}_k= p$, it yields
\begin{eqnarray*}
\Big(E\|Z^{\varepsilon}(s_k\wedge \tau)\|^{\tilde{r}_k}_{L^{\tilde{r}_k}}\Big)^{\frac{1}{\tilde{r}_k}}\leq \varepsilon^{\frac{1}{2}} C^{\frac{1}{\tilde{r}_k}}_2,
\end{eqnarray*}
which implies
\begin{eqnarray*}
E\prod^p_{k=1}\|Z^{\varepsilon}(s_k\wedge \tau)\|_{L^{p_1}}
\leq \varepsilon^{\frac{p}{2}}C_2.
\end{eqnarray*}
Hence,
\begin{eqnarray*}
K_1&\leq& \varepsilon^{\frac{p}{2}} K^p(1+\sqrt{C_0})^pC(C_2)\Big(\int^{t_2}_0\frac{t_1-t_2}{(t_1-s)\sqrt{t_2-s}}(t_2-s)^{-\frac{1+{\alpha_0}}{2(2+{\alpha_0})}}ds \Big)^p \\
&\leq& \varepsilon^{\frac{p}{2}}C(K,p,C_0,C_2)|t_1-t_2|^{\frac{p}{2(2+{\alpha_0})}}\\
&\leq& \varepsilon^{\frac{p}{2}}C(K,p,C_0,C_2)|t_1-t_2|^{\frac{p}{6}}.
\end{eqnarray*}
Indeed, let $u=t_1-s, v=\frac{u}{t_1-t_2}$, it follows that
\begin{eqnarray}\notag
&&\int^{t_2}_0\frac{t_1-t_2}{(t_1-s)\sqrt{t_2-s}}(t_2-s)^{-\frac{1+{\alpha_0}}{2(2+{\alpha_0})}}ds \\ \notag
&=& \int^{t_1-t_2}_{t_1}\frac{t_1-t_2}{u}(u-t_1+t_2)^{-\frac{1}{2}-\frac{1+{\alpha_0}}{2(2+{\alpha_0})}}du\\ \notag
&=& \int^{\frac{t_1}{t_1-t_2}}_{1}\frac{1}{v}(v-1)^{-\frac{1}{2}-\frac{1+{\alpha_0}}{2(2+{\alpha_0})}}(t_1-t_2)^{\frac{1}{2}-\frac{1+{\alpha_0}}{2(2+{\alpha_0})}}dv\\
\notag
&\leq& |t_1-t_2|^{\frac{1}{2}-\frac{1+{\alpha_0}}{2(2+{\alpha_0})}}\int^{\infty}_1\frac{1}{v}(v-1)^{-\frac{1}{2}-\frac{1+{\alpha_0}}{2(2+{\alpha_0})}}dv\\
\label{eee-16}
&\leq& C|t_1-t_2|^{\frac{1}{2(2+{\alpha_0})}}.
\end{eqnarray}
By H\"{o}lder inequality, (\ref{eqq-5-1}),  Lemma \ref{lem-4}, for $p>14$ and for some $0<\alpha_1<1, 0<\kappa<1$, we get
\begin{eqnarray*}
K_2&\leq&E\Big[\int^1_0\Big|K\int^{t_2}_0\int^1_0\frac{1}{\sqrt{t_1-s}}(\tilde{G}_{t_1-s}(x,y)-\tilde{G}_{t_2-s}(x,y))(1+|u^0|)|Z^{\varepsilon}(s\wedge\tau)|dsdy\Big|^2dx\Big]^{\frac{p}{2}}\\
&\leq&K^p(1+\sqrt{C_0})^pE\Big[\int^1_0\Big|\int^{t_2}_0\frac{1}{\sqrt{t_1-s}}\Big(\int^1_0(\tilde{G}_{t_1-s}(x,y)-\tilde{G}_{t_2-s}(x,y))^{3-{\alpha_1}}dy\Big)^{\frac{1}{3-{\alpha_1}}}
\|Z^{\varepsilon}(s\wedge\tau)\|_{L^{p}}ds\Big|^2dx\Big]^{\frac{p}{2}}\\
&\leq&K^p(1+\sqrt{C_0})^p E\Big[\int^1_0\big(\int^{t_2}_0(t_1-s)^{-\kappa}ds\big)\Big(\int^{t_2}_0
\int^1_0(\tilde{G}_{t_1-s}(x,y)-\tilde{G}_{t_2-s}(x,y))^{3-{\alpha_1}}dyds\Big)^{\frac{2}{3-{\alpha_1}}}
\\
&& \times\Big(\int^{t_2}_0(t_1-s)^{-\frac{1}{2}(1-\kappa)p}\|Z^{\varepsilon}(s\wedge\tau)\|^p_{L^{p}}ds\Big)^{\frac{2}{p}}dx\Big]^{\frac{p}{2}}\\
&\leq& K^p(1+\sqrt{C_0})^p \Big(\int^{t_2}_0(t_1-s)^{-\kappa}ds\Big)^{\frac{p}{2}}\sup_{x\in[0,1]}\Big(\int^{t_2}_0
\int^1_0(\tilde{G}_{t_1-s}(x,y)-\tilde{G}_{t_2-s}(x,y))^{3-{\alpha_1}}dyds\Big)^{\frac{p}{3-{\alpha_1}}}
\\
&& \times \Big(\int^{t_2}_0(t_1-s)^{-\frac{1}{2}(1-\kappa)p}E\|Z^{\varepsilon}(s\wedge\tau)\|^p_{L^{p}}ds\Big)\\
&\leq& \varepsilon^{\frac{p}{2}}K^p(1+\sqrt{C_0})^p C_2 \Big(\int^{t_2}_0(t_1-s)^{-\kappa}ds\Big)^{\frac{p}{2}}\sup_{x\in[0,1]}\Big(\int^{t_2}_0
\int^1_0(\tilde{G}_{t_1-s}(y,s)-\tilde{G}_{t_2-s}(y,s))^{3-{\alpha_1}}dyds\Big)^{\frac{p}{3-{\alpha_1}}}
\\
&& \times\Big(\int^{t_2}_0(t_1-s)^{-\frac{1}{2}(1-\kappa)p}ds\Big),
\end{eqnarray*}
where $\frac{1}{3-\alpha_1}+\frac{1}{p}=\frac{1}{2}$.

Define $c_0:=-\frac{1}{2}(1-\kappa)p+1$, then
\begin{eqnarray*}
K_2
&\leq&\varepsilon^{\frac{p}{2}}K^p(1+\sqrt{C_0})^pC_2(|t_1-t_2|^{-\kappa+1}-t^{-\kappa+1}_1)^{\frac{p}{2}}|t_1-t_2|^{\frac{{\alpha_1} p}{2(3-{\alpha_1})}}[t^{c_0}_1-|t_1-t_2|^{c_0}]\\
&\leq& \varepsilon^{\frac{p}{2}}C(K,p,C_0, T,C_2)|t_1-t_2|^{\frac{{\alpha_1} p}{2(3-{\alpha_1})}},
\end{eqnarray*}
which implies that
\begin{eqnarray}\label{eee-17}
K_2\leq \varepsilon^{\frac{p}{2}}C(K,p,C_0, T,C_2)|t_1-t_2|^{\frac{p-6}{4}}.
\end{eqnarray}
Combing (\ref{eee-16}) and (\ref{eee-17}), it yields
\begin{eqnarray}\label{eee-18}
E\|\bar{I}^{\varepsilon}_{2,3}\|^p_H\leq  C(K,p,C_0,C_2)|t_1-t_2|^{\frac{p}{6}}+C(K,p,C_0, T,C_2)|t_1-t_2|^{\frac{p-6}{4}}.
\end{eqnarray}
For any $p>14$ and $\frac{1}{r}+\frac{2}{p}=1$, we have $r\in (1,\frac{7}{6})$. Utilizing (\ref{eq-29}) and (\ref{eq-36}), we have
\begin{eqnarray*}
E\|\bar{I}^{\varepsilon}_{2,4}\|^p_H&\leq&CK^pE\Big[\int^1_0|\int^{t_2}_0\int^1_0\partial_y(G_{t_1-s}(x,y)-G_{t_2-s}(x,y))|Y^{\varepsilon}|^2I_{\{s\leq \tau\}}dsdy|^2dx\Big]^{\frac{p}{2}}\\
&\leq& CK^pE\Big[\int^1_0\Big(\int^{t_2}_0\int^1_0|\partial_y(G_{t_1-s}(x,y)-G_{t_2-s}(x,y))|^rdyds\Big)^{\frac{2}{r}}\Big(\int^{t_2}_0\int^1_0|Y^{\varepsilon}(s\wedge \tau)|^{p}dsdy\Big)^{\frac{4}{p}}dx\Big]^{\frac{p}{2}}\\
&\leq& CK^p|t_1-t_2|^{(\frac{3}{2}-r)\frac{p}{r}}E\Big[\int^1_0\Big(\int^{t_2}_0\int^1_0|Y^{\varepsilon}(s\wedge \tau,y)|^{2p}dyds\Big)^{\frac{2}{p}}\Big(\int^{t_2}_0\int^1_0dyds\Big)^{\frac{2}{p}}\Big]^{\frac{p}{2}}\\
&\leq& C(T)K^p|t_1-t_2|^{(\frac{3}{2}-r)\frac{p}{r}}E\int^{t_2}_0\int^1_0|Y^{\varepsilon}(s\wedge \tau,y)|^{2p}dyds\\
&\leq& C(T,K,p,C_1)|t_1-t_2|^{(\frac{3}{2}-r)\frac{p}{r}},
\end{eqnarray*}
which implies
\begin{eqnarray}\label{eee-19}
E\|\bar{I}^{\varepsilon}_{2,4}\|^p_H\leq C(T,K,p,C_1)|t_1-t_2|^{\frac{p-6}{2}}.
\end{eqnarray}

Combing  (\ref{eee-14}),  (\ref{eee-15}),  (\ref{eee-18}) and (\ref{eee-19}), we conclude that
\begin{eqnarray*}
E\|\bar{I}^{\varepsilon}_2(t_1)-\bar{I}^{\varepsilon}_2(t_2)\|^p_H&\leq& \varepsilon^{\frac{p}{2}}C_p( E\|\bar{I}^{\varepsilon}_{2,1}\|^p_H+E\|\bar{I}^{\varepsilon}_{2,2}\|^p_H+E\|\bar{I}^{\varepsilon}_{2,3}\|^p_H+E\|\bar{I}^{\varepsilon}_{2,4}\|^p_H)\\
&\leq& \varepsilon^{\frac{p}{2}}C_p\Big[C(K,p, C_0,C_2)|t_1-t_2|^{\frac{p-2}{4}}+ C(p,K,C_1)|t_1-t_2|^{\frac{p-2}{2}}\\
&& +C(K,p,C_0,T,C_2)|t_1-t_2|^{\frac{p}{6}}+C(K,p,C_0,T,C_2)|t_1-t_2|^{\frac{p-6}{4}}\\
&&+C(T,K,p,C_1)|t_1-t_2|^{\frac{p-6}{2}}\Big]\\
&\leq& \varepsilon^{\frac{p}{2}}C(K,p, T,C_0,C_1,C_2)(|t_1-t_2|^{\frac{p}{6}}+|t_1-t_2|^{\frac{p-6}{4}}).
\end{eqnarray*}

As $p>14$, we get $\frac{p}{6}>\frac{p-6}{4}>2$, then
\begin{eqnarray}\label{eq-39}
E\|\bar{I}^{\varepsilon}_2(t_1)-\bar{I}^{\varepsilon}_2(t_2)\|^{p}_H\leq \varepsilon^{\frac{p}{2}}C(K,p, T,C_0,C_1,C_2)|t_1-t_2|^{\frac{p-6}{4}}.
\end{eqnarray}
Let
\[
\Psi(r)=r^p,\quad p(r)=r^{\frac{p-6}{4p}},
\]
and
\[
\rho=\int^T_0\int^T_0\Big|\frac{\|\bar{I}^{\varepsilon}_2(t_1)-\bar{I}^{\varepsilon}_2(t_2)\|_H}{|t_1-t_2|^{\frac{p-6}{4p}}}\Big|^pdt_1dt_2.
\]
By Lemma \ref{lem-3}, for any $s,t\in [0,T]$,we have
\begin{eqnarray*}
\|\bar{I}^{\varepsilon}_2(t_1)-\bar{I}^{\varepsilon}_2(t_2)\|_H&\leq& 8\int^{|t_1-t_2|}_0(\rho r^{-2})^{\frac{1}{p}}dr^{\frac{p-6}{4p}}\\
&\leq& C(p)\rho^{\frac{1}{p}}\int^{|t_1-t_2|}_0r^{-\frac{2}{p}+\frac{p-6}{4p}-1}dr.
\end{eqnarray*}
As $\frac{p-6}{4}>2$, we have $-\frac{2}{p}+\frac{p-6}{4p}>0$, then
\begin{eqnarray*}
\|\bar{I}^{\varepsilon}_2(t_1)-\bar{I}^{\varepsilon}_2(t_2)\|_H\leq
C(p)\rho^{\frac{1}{p}}|t_1-t_2|^{-\frac{2}{p}+\frac{p-6}{4p}}.
\end{eqnarray*}
Taking $t_2=0, 0\leq t=t_1\leq T$, we get
\begin{eqnarray*}
\|\bar{I}^{\varepsilon}_2(t)\|_H\leq
C(T,p)\rho^{\frac{1}{p}}.
\end{eqnarray*}
By (\ref{eq-39}), we get
 \begin{eqnarray}\label{eq-14-1}
E\sup_{0\leq t\leq T}\|\bar{I}^{\varepsilon}_2(t)\|_H\leq C(T,p)(E\rho)^{\frac{1}{p}}\leq \sqrt{\varepsilon}C(K,p, T,C_0,C_1,C_2).
\end{eqnarray}
Thus, we deduce from (\ref{eq-30}) that
 \begin{eqnarray}\notag
E\sup_{0\leq t\leq T}\|{I}^{\varepsilon}_2(t)\|_H&=&E\sup_{0\leq t\leq T}\|\bar{I}^{\varepsilon}_2(t\wedge \tau)\|_H\\
\label{eq-14}
&\leq & \sqrt{\varepsilon}C(K,p, T,C_0,C_1,C_2)\rightarrow 0, \ {\rm{as}}\ \varepsilon\rightarrow 0.
\end{eqnarray}

By the same argument as the method dealing with $I^{\varepsilon}_2(t)$, we get
\begin{eqnarray}\label{eee-24}
E\sup_{t\in [0,T]}\|I^{\varepsilon}_1(t)\|_H\rightarrow 0, \ {\rm{as}} \ \varepsilon\rightarrow 0.
\end{eqnarray}

Combing (\ref{eq-14}) and (\ref{eee-24}),  we know that $I^{\varepsilon}_1(t)+I^{\varepsilon}_2(t)\rightarrow 0$ in probability in $C([0, T]; H)$ as $\varepsilon\rightarrow 0$. By (\ref{eq-27}),
we complete the proof.

\section{MDP for semilinear SPDE}\label{sec-3}
In this part, we are concerned with the moderate deviation principle of the solution $u^\varepsilon$ of (\ref{eqq-9}). As stated in the introduction, we need to
prove $\frac{u^\varepsilon-u^0}{\sqrt{\varepsilon}\lambda(\varepsilon)}$ satisfies a large deviation principle on $C([0,T]; H)$ with $\lambda(\varepsilon)$ satisfying (\ref{e-43}).
From now on, we assume (\ref{e-43}) holds.
\subsection{The weak convergence approach}
Let $X^\varepsilon =\frac{u^\varepsilon-u^0}{\sqrt{\varepsilon}\lambda(\varepsilon)}$, we will use the weak convergence approach introduced by Budhiraja and Dupuis in \cite{BD} to verify $X^\varepsilon$ satisfies a large deviation principle. Firstly, we recall some standard definitions and results from the large deviation theory (see \cite{DZ}).

Let $\{X^\varepsilon\}$ be a family of random variables defined on a probability space $(\Omega, \mathcal{F}, P)$ taking values in some Polish space $\mathcal{E}$.

\begin{dfn}
(Rate function) A function $I: \mathcal{E}\rightarrow [0,\infty]$ is called a rate function if $I$ is lower semicontinuous. A rate function $I$ is called a good rate function if the level set $\{x\in \mathcal{E}: I(x)\leq M\}$ is compact for each $M<\infty$.
\end{dfn}
\begin{dfn}
 (LDP) The sequence $\{X^{\varepsilon}\}$ is said to satisfy the large deviation principle with rate function $I$ if for each Borel subset $A$ of $\mathcal{E}$
      \[
      -\inf_{x\in A^o}I(x)\leq \lim \inf_{\varepsilon\rightarrow 0}\varepsilon \log P(X^{\varepsilon}\in A)\leq \lim \sup_{\varepsilon\rightarrow 0}\varepsilon \log P(X^{\varepsilon}\in A)\leq -\inf_{x\in \bar{A}}I(x).
      \]
\end{dfn}

Now we define
\begin{eqnarray*}
\mathcal{A}&=&\Big\{\phi: \int^T_0\int^1_0 |\phi(s,y)|^2dyds<\infty\quad  P\text{-}a.s.\Big\};\\
T_M&=&\Big\{ h\in L^2([0,T]\times [0,1]): \int^T_0\int^1_0 |h(s,y)|^2dyds\leq M\Big\};\\
\mathcal{A}_M&=&\Big\{\phi\in \mathcal{A}: \phi(\omega)\in T_M,\ P\text{-}a.s.\Big\}.
\end{eqnarray*}
Here and in the sequel of this paper, we will always refer to the weak topology on the set  $T_M$, in this case, $T_M$ is a compact metric space of $L^2([0,T]\times [0,1])$.

Suppose $\mathcal{G}^{\varepsilon}: C([0,T]\times [0,1]; \mathbb{R})\rightarrow \mathcal{E}$ is a measurable mapping and $X^{\varepsilon}=\mathcal{G}^{\varepsilon}(W)$. Now, we list below sufficient conditions for the large deviation principle of the sequence $X^{\varepsilon}$ as $\varepsilon\rightarrow 0$.
\begin{description}
  \item[\textbf{Hypothesis G} ] There exists a measurable map $\mathcal{G}^0: C([0,T]\times [0,1]; \mathbb{R})\rightarrow \mathcal{E}$ satisfying
\end{description}
\begin{description}
  \item[(i)] For every $M<\infty$, let $\{h^{\varepsilon}: \varepsilon>0\}$ $\subset \mathcal{A}_M$. If $h^{\varepsilon}$ converges to $h$ as $T_M-$valued random elements in distribution, then $\mathcal{G}^{\varepsilon}\Big(W(\cdot)+\lambda(\varepsilon)\int^{\cdot}_{0}\int^{\cdot}_{0}h^\varepsilon(s,y)dyds\Big)$ converges in distribution to $\mathcal{G}^0\Big(\int^{\cdot}_{0}\int^{\cdot}_{0}h(s,y)dyds\Big)$.
  \item[(ii)] For every $M<\infty$, $K_M=\Big\{\mathcal{G}^0\Big(\int^{\cdot}_{0}\int^{\cdot}_{0}h(s,y)dyds\Big): h\in T_M\Big\}$ is a compact subset of $\mathcal{E}$.
\end{description}
The following result is due to Budhiraja et al. in \cite{BD}.

\begin{thm}\label{thm-3}
If $\mathcal{G}^{0}$ satisfies Hypothesis G, then $X^{\varepsilon}$ satisfies a large deviation principle on $\mathcal{E}$ with the good rate function $I$ given by
\begin{eqnarray}\label{eq-5}
I(f)=\inf_{\Big\{h\in L^2([0,T]\times [0,1]): f= \mathcal{G}^0(\int^{\cdot}_{0}\int^{\cdot}_{0}h(s,y)dyds)\Big\}}\Big\{\frac{1}{2}\int^T_0\int^T_0|h(s,y)|^2dyds\Big\},\ \ \forall f\in\mathcal{E}.
\end{eqnarray}
By convention, $I(\emptyset)=\infty$.
\end{thm}
In this part, we are concerned with the following SPDE driven by small multiplicative noise
\begin{eqnarray}\notag
&&X^{\varepsilon}(x,t)\\
\notag
&=&\frac{1}{\lambda(\varepsilon)}\int^t_0\int^1_0G_{t-s}(x,y)\sigma(s,y, u^0(s,y)+\sqrt{\varepsilon}\lambda(\varepsilon)X^{\varepsilon}(s,y))W(dyds)\\ \notag
&& +\frac{1}{\sqrt{\varepsilon}\lambda(\varepsilon)}\int^t_0\int^1_0G_{t-s}(x,y)(b(s,y,u^0(s,y)+\sqrt{\varepsilon}\lambda(\varepsilon)X^{\varepsilon}(s,y))-b(s,y,u^0(s,y)))dyds\\
\label{eqq-16}
&& -\frac{1}{\sqrt{\varepsilon}\lambda(\varepsilon)}\int^t_0\int^1_0\partial_yG_{t-s}(x,y)(g(s,y,u^0(s,y)+\sqrt{\varepsilon}\lambda(\varepsilon)X^{\varepsilon}(s,y))-g(s,y,u^0(s,y)))dyds.
\end{eqnarray}
Under (H1) and (H2), combing Theorem \ref{thm-1} and Lemma \ref{lem-5}, there exists a unique strong solution in $X^\varepsilon\in C([0,T];H)$.
Therefore, there exists a Borel-measurable function
\begin{eqnarray}\label{e-46}
\mathcal{G}^{\varepsilon}: C([0,T]\times [0,1]; \mathbb{R})\rightarrow C([0,T];H)
\end{eqnarray}
such that $X^{\varepsilon}(\cdot)=\mathcal{G}^{\varepsilon}(W(\cdot))$.

Let $h\in T_M$, consider the following skeleton equation
\begin{eqnarray}\notag
X^h(x,t)&=&\int^t_0\int^1_0G_{t-s}(x,y)\partial_rb(s,y,u^0(s,y))X^h(s,y)dyds\\ \notag
&& -\int^t_0\int^1_0\partial_yG_{t-s}(x,y)\partial_rg(s,y,u^0(s,y))X^h(s,y)dyds\\
\label{eqq-17}
&& +\int^t_0\int^1_0G_{t-s}(x,y)\sigma(s,y,u^0(s,y))h(s,y)dyds.
\end{eqnarray}
By (H3), we know that all coefficients of (\ref{eqq-17}) are Lipschitz, it admits a unique solution $X^h$ satisfying
\begin{eqnarray}\label{eee-26}
\sup_{t\in [0,T]}\|X^h(t)\|^2_H\leq C(K,T,M,C_0).
\end{eqnarray}
Therefore, we can define a measurable mapping $\mathcal{G}^0: C([0,T]\times [0,1]; \mathbb{R})\rightarrow C([0,T];H)$ such that $\mathcal{G}^0\Big(\int^{\cdot}_0\int^{\cdot}_0 h(s,y)dyds\Big):=X^h(\cdot)$.

\smallskip

The main result in this part reads as
\begin{thm}\label{thm-8}
Let the initial value $f\in L^p([0,1])$ for all $p\in [2,\infty)$. Under (H1)-(H3), $X^{\varepsilon}$ satisfies a large deviation principle on $C([0,T];H)$ with the good rate function $I$ defined by (\ref{eq-5}).
\end{thm}

\subsection{Tightness of semilinear SPDE with small perturbations}
For any  $h^{\varepsilon}\in \mathcal{A}_M$, consider the following SPDE
\begin{eqnarray}\notag
&&\bar{X}^{\varepsilon,h^{\varepsilon}}(x,t)\\ \notag
&=&\frac{1}{\lambda(\varepsilon)}\int^t_0\int^1_0G_{t-s}(x,y)\sigma(s,y, u^0(s,y)+\sqrt{\varepsilon}\lambda(\varepsilon)\bar{X}^{\varepsilon,h^{\varepsilon}}(s,y))W(dyds)\\ \notag
&&+\int^t_0\int^1_0G_{t-s}(x,y)\sigma(s,y, u^0(s,y)+\sqrt{\varepsilon}\lambda(\varepsilon)\bar{X}^{\varepsilon,h^{\varepsilon}}(s,y))h^{\varepsilon}(s,y)dyds\\ \notag
&&+\frac{1}{\sqrt{\varepsilon}\lambda(\varepsilon)}\int^t_0\int^1_0G_{t-s}(x,y)(b(s,y,u^0(s,y)+\sqrt{\varepsilon}\lambda(\varepsilon)\bar{X}^{\varepsilon,h^{\varepsilon}}(s,y))-b(s,y,u^0(s,y)))dyds\\
\label{eqq-18}
&&-\frac{1}{\sqrt{\varepsilon}\lambda(\varepsilon)}\int^t_0\int^1_0\partial_yG_{t-s}(x,y)(g(s,y,u^0(s,y)+\sqrt{\varepsilon}\lambda(\varepsilon)\bar{X}^{\varepsilon,h^{\varepsilon}}(s,y))-g(s,y,u^0(s,y)))dyds.
\end{eqnarray}
with $\bar{X}^{\varepsilon,h^{\varepsilon}}(0)=0$, then  $\mathcal{G}^{\varepsilon}\Big(W(\cdot)+\lambda(\varepsilon)\int^{\cdot}_0\int^{\cdot}_0h^{\varepsilon}(s,y)dyds\Big)=\bar{X}^{\varepsilon,h^{\varepsilon}}$.

Moreover, with the aid of Lemma \ref{lem-5} and by using the same method as Theorem 2.1 in \cite{G98}, it follows that
\begin{lemma}\label{lem-6}
For any family $\{h^\varepsilon; \varepsilon>0\}\subset \mathcal{A}_M$, it holds that
\begin{eqnarray}\label{eqqq-6}
\lim_{C\rightarrow \infty}\sup_{0<\varepsilon\leq 1}P\Big(\sup_{t\in[0,T]}\|\bar{X}^{\varepsilon,h^{\varepsilon}}(t)\|^2_H>C\Big)=0.
\end{eqnarray}
\end{lemma}

Referring to \cite{FS17}, the following lemma gives a criterion to ensure tightness.
\begin{lemma}\label{lemm-2}
Let $\rho\in [1,\infty)$, and $q\in [1, \rho)$. Let $\zeta_n(t,y)$ be a sequence of  random fields on $[0,T]\times [0,1]$ such that $\sup_{0\leq t\leq T}\|\zeta_n(t,\cdot)\|_{L^q}\leq \theta_n$, where $\theta_n$ is a finite random variable for every $n$. Assume that the sequence $\theta_n$ is bounded in probability, i.e.,
$\lim_{C\rightarrow \infty}\sup_nP(\theta_n\geq C)=0$.
Then the sequence
\[
J(\zeta_n):=\int^t_0\int^1_0R(s,t;x,y)\zeta_n(r,y)dyds, t\in[0,T], x\in [0,1],
 \]
 where $R(s,t;x,y)=\partial_yG(s,t;x,y)$ or $R(s,t;x,y)=G(s,t;x,y)$ is uniformly tight in $C([0,T];L^{\rho}([0,1]))$.
\end{lemma}

Let $\mathcal{D}(X)$ be the distribution of a random variable $X$.
\begin{prp}\label{prp-6}
For any $R>0$, $\mathcal{D}(\bar{X}^{\varepsilon,h^{\varepsilon}})_{\varepsilon\in (0,1]}$ is tight in $C([0,T];H).$
\end{prp}
\begin{proof}
From (\ref{eqq-18}), we have
\begin{eqnarray}\notag
&&\bar{X}^{\varepsilon,h^{\varepsilon}}(x,t)\\ \notag
&=&\frac{1}{\lambda(\varepsilon)}\int^t_0\int^1_0G_{t-s}(x,y)\sigma(s,y, u^0(s,y)+\sqrt{\varepsilon}\lambda(\varepsilon)\bar{X}^{\varepsilon,h^{\varepsilon}}(s,y))W(dyds)\\ \notag
&&+\int^t_0\int^1_0G_{t-s}(x,y)\sigma(s,y, u^0(s,y)+\sqrt{\varepsilon}\lambda(\varepsilon)\bar{X}^{\varepsilon,h^{\varepsilon}}(s,y))h^{\varepsilon}(s,y)dyds\\ \notag
&&+\frac{1}{\sqrt{\varepsilon}\lambda(\varepsilon)}\int^t_0\int^1_0G_{t-s}(x,y)(b(s,y,u^0(s,y)+\sqrt{\varepsilon}\lambda(\varepsilon)\bar{X}^{\varepsilon,h^{\varepsilon}}(s,y))-b(s,y,u^0(s,y)))dyds\\
\notag
&&-\frac{1}{\sqrt{\varepsilon}\lambda(\varepsilon)}\int^t_0\int^1_0\partial_yG_{t-s}(x,y)(g(s,y,u^0(s,y)+\sqrt{\varepsilon}\lambda(\varepsilon)\bar{X}^{\varepsilon,h^{\varepsilon}}(s,y))-g(s,y,u^0(s,y)))dyds\\
\label{eqqq-5}
&:=&J^{\varepsilon}_1+J^{\varepsilon}_2+J^{\varepsilon}_3+J^{\varepsilon}_4.
\end{eqnarray}
Firstly, we claim that
\begin{eqnarray}\label{equat-2}
\lim_{\varepsilon\rightarrow 0}E\sup_{0\leq t\leq T}\|J^{\varepsilon}_1\|_H= 0.
\end{eqnarray}
Indeed, by (H1) and using the similar method as the estimation of (\ref{eq-17}), for any $t_1, t_2\in [0,T]$ and $p>14$, we obtain
\begin{eqnarray}\label{eq-28}
E\|J^{\varepsilon}_1(t_1)-J^{\varepsilon}_1(t_2)\|^p_H
\leq \frac{C(K,p)}{(\lambda(\varepsilon))^{p}}|t_1-t_2|^{\frac{p}{4}}.
\end{eqnarray}
Applying Lemma \ref{lem-3} with
\[
\Psi(r)=r^{p}, \quad p(r)=r^{\frac{1}{4}},
\]
and
\begin{eqnarray*}
\varrho=\int^T_0\int^T_0\left(\frac{\|J^{\varepsilon}_1(t_1)-J^{\varepsilon}_1(t_2)\|_H}{|t_1-t_2|^{\frac{1}{4}}}\right)^{p}dt_1dt_2,
\end{eqnarray*}
we have
\begin{eqnarray*}
\|J^{\varepsilon}_1(t_1)-J^{\varepsilon}_1(t_2)\|_H&\leq&8\int^{|t_1-t_2|}_0(\varrho r^{-2})^{\frac{1}{p}}dr^{\frac{1}{4}}\\
&\leq&C\varrho^{\frac{1}{p}}\int^{|t_1-t_2|}_0r^{-\frac{3}{4}-\frac{2}{p}}dr\\
&\leq&C\varrho^{\frac{1}{p}}|t_1-t_2|^{\frac{p-8}{4p}}.
\end{eqnarray*}
Let $t=t_1$ and $t_2=0$, we get
\begin{eqnarray*}
\|J^{\varepsilon}_1(t)\|_H\leq C\varrho^{\frac{1}{p}}t^{\frac{p-8}{4p}},
\end{eqnarray*}
which implies that
\begin{eqnarray*}
\sup_{t\in [0,T]}\|J^{\varepsilon}_1(t)\|_H\leq C(T)\varrho^{\frac{1}{p}}.
\end{eqnarray*}
By (\ref{eq-28}), it gives that
\[
E\varrho\leq \frac{C(L,p)}{(\lambda(\varepsilon))^{p}}\rightarrow 0, \quad \varepsilon\rightarrow 0,
\]
Thus,
\begin{eqnarray*}
E\sup_{t\in [0,T]}\|J^{\varepsilon}_1(t)\|_H\leq C(T)E\varrho\rightarrow 0, \  {\rm{as}} \ \varepsilon\rightarrow 0.
\end{eqnarray*}
By (\ref{equat-2}), we deduce that $J^{\varepsilon}_1$ converges in probability in $C([0,T];H)$.

By (H1), we get
\begin{eqnarray}\notag
\sup_{0<\varepsilon\leq 1}J^{\varepsilon}_2&=&\sup_{0<\varepsilon\leq 1}\int^t_0\int^1_0G_{t-s}(x,y)\sigma(s,y, u^0(s,y)+\sqrt{\varepsilon}\lambda(\varepsilon)\bar{X}^{\varepsilon,h^{\varepsilon}}(s,y))h^{\varepsilon}(s,y)dyds\\
\notag
&\leq& K\Big(\int^t_0\int^1_0G^2_{t-s}(x,y)dyds\Big)^{\frac{1}{2}}\sup_{0<\varepsilon\leq 1}\Big(\int^t_0\int^1_0|h^{\varepsilon}(s,y)|^2dyds\Big)^{\frac{1}{2}}\\
\label{equat-4}
&\leq& KC(T)M^{\frac{1}{2}}.
\end{eqnarray}
Referring to (4.2) in \cite{FS17}, (\ref{equat-4}) implies the tightness of $J^{\varepsilon}_2$.

For $J^{\varepsilon}_3$, applying Lemma \ref{lemm-2} with $\rho=2, q=1$, by (H2), we have
\begin{eqnarray*}
&&\sup_{0\leq t\leq T}\|b(s,y,u^0+\sqrt{\varepsilon}\lambda(\varepsilon)\bar{X}^{\varepsilon,h^{\varepsilon}}(s,y))-b(s,y,u^0(s,y))\|_{L^1}\\
&\leq& L\sup_{0\leq t\leq T}\sqrt{\varepsilon}\lambda(\varepsilon)\int^1_{0}(1+|u^0|+\sqrt{\varepsilon}\lambda(\varepsilon)|\bar{X}^{\varepsilon,h^{\varepsilon}}|)|\bar{X}^{\varepsilon,h^{\varepsilon}}|dx\\
&\leq&L \sup_{0\leq t\leq T}\sqrt{\varepsilon}\lambda(\varepsilon)[(1+\|u^0\|_H+\sqrt{\varepsilon}\lambda(\varepsilon)\|\bar{X}^{\varepsilon,h^{\varepsilon}}\|_H)\|\bar{X}^{\varepsilon,h^{\varepsilon}}\|_H]\\
&\leq& \sqrt{\varepsilon}\lambda(\varepsilon)L[1+C_0+(1+\sqrt{\varepsilon}\lambda(\varepsilon))\sup_{0\leq t\leq T}\|\bar{X}^{\varepsilon,h^{\varepsilon}}\|^2_H].
\end{eqnarray*}
Let
\[
\theta=\sqrt{\varepsilon}\lambda(\varepsilon)L\Big(1+C_0+(1+\sqrt{\varepsilon}\lambda(\varepsilon))\sup_{0\leq t\leq T}\|\bar{X}^{\varepsilon,h^{\varepsilon}}\|^2_H\Big),
\]
we have
\begin{eqnarray*}
&&\lim_{M\rightarrow \infty} \sup_{0<\varepsilon\leq 1} P(\theta\geq M)\\
&\leq& \lim_{M\rightarrow \infty} \sup_{0<\varepsilon\leq 1}P\Big(\sqrt{\varepsilon}\lambda(\varepsilon)L(1+C_0)\geq \frac{M}{2}\Big)\\
&& + \lim_{M\rightarrow \infty} \sup_{0<\varepsilon\leq 1}P\Big(\sup_{0\leq t\leq T}\|\bar{X}^{\varepsilon,h^{\varepsilon}}\|^2_H\geq \frac{M}{2\sqrt{\varepsilon}\lambda(\varepsilon)L(1+\sqrt{\varepsilon}\lambda(\varepsilon))}\Big).
\end{eqnarray*}
By Lemma \ref{lem-6}, we obtain
\begin{eqnarray}
\lim_{M\rightarrow \infty} \sup_{0<\varepsilon\leq 1} P(\theta\geq M)=0.
\end{eqnarray}
Thus, we get the tightness of $J^{\varepsilon}_3$ in $C([0,T];H)$.
Employing similar method as above, we obtain the tightness of $J^{\varepsilon}_4$ in $C([0,T];H)$.
We complete the proof.
\end{proof}

\subsection{The proof of MDP }

According to Theorem \ref{thm-2}, the proof of MDP  will be completed if the following Theorem \ref{thm-5} and Theorem \ref{thm-4} are established.
\begin{thm}\label{thm-5}
The family
\[
K_M=\left\{\mathcal{G}^0(\int^{\cdot}_{0}\int^{\cdot}_{0}h(s,y)dyds): h\in T_M\right\}
\]
is a compact subset of $C([0,T];H)$.
\end{thm}
\begin{proof}
Let $X^{h_n}=\{\mathcal{G}^0(\int^{\cdot}_0\int^{\cdot}_0h_n(s,y)dyds): n\geq 1\}$ be a sequence of $K_M$. Due to the fact that $T_M$ is a compact subset of $L^2([0,T]\times[0,1])$ under weak topology, there exists a subsequence still denoted by $\{n\}$ and an element $h\in T_M$ such that
 $h_{n}\rightarrow h$ weakly in $T_M$ as $n \rightarrow \infty$.
We need to prove $X^{h_n}\rightarrow X^h$ strongly in $C([0,T];H)$.

From (\ref{eqq-17}), we know that
\begin{eqnarray*}
 X^{h_n}(t,x)-X^h(t,x)&=& \int^t_0\int^1_0G_{t-s}(x,y)\partial_rb(s,y,u^0(s,y))( X^{h_n}(s,y)-X^h(s,y))dyds\\
 && -\int^t_0\int^1_0\partial_yG_{t-s}(x,y)\partial_rg(s,y,u^0(s,y))( X^{h_n}(s,y)-X^h(s,y))dyds\\
 && +\int^t_0\int^1_0G_{t-s}(x,y)\sigma(s,y,u^0(s,y))(h_n(s,y)-h(s,y))dyds\\
 &:=&J^n_1(t)+J^n_2(t)+J^n_3(t).
\end{eqnarray*}
By Lemma \ref{lem-1}, (H3) and Lemma \ref{lem-5}, we deduce that
\begin{eqnarray}\notag
 \|J^n_1(t)\|_H&\leq& C\int^t_0(t-s)^{-\frac{3}{4}}\|\partial_rb(s,y,u^0(s,y))( X^{h_n}(s,y)-X^h(s,y))\|_{L^1}ds\\ \notag
 &\leq& CK\int^t_0(t-s)^{-\frac{3}{4}}\|(1+u^0(s,y))( X^{h_n}(s,y)-X^h(s,y))\|_{L^1}ds\\ \notag
 &\leq& CK\int^t_0(t-s)^{-\frac{3}{4}}(1+\|u^0(s)\|_H)\| X^{h_n}(s)-X^h(s)\|_{H}ds\\
 \label{ee-6}
 &\leq& CK(1+C_0)\int^t_0(t-s)^{-\frac{3}{4}}\| X^{h_n}(s)-X^h(s)\|_{H}ds.
 \end{eqnarray}
 Similar to $J^n_1(t)$, we deduce that
 \begin{eqnarray}\notag
 \|J^n_2(t)\|_H&\leq& CK\int^t_0(t-s)^{-\frac{3}{4}}\|\partial_rg(s,y,u^0(s,y))( X^{h_n}(s,y)-X^h(s,y))\|_{L^1}ds\\ \notag
 &\leq& CK\int^t_0(t-s)^{-\frac{3}{4}}\|(1+u^0(s,y))( X^{h_n}(s,y)-X^h(s,y))\|_{L^1}ds\\ \notag
 &\leq& CK\int^t_0(t-s)^{-\frac{3}{4}}(1+\|u^0(s)\|_H)\| X^{h_n}(s)-X^h(s)\|_{H}ds\\
 \label{ee-7}
 &\leq& CK(1+C_0)\int^t_0(t-s)^{-\frac{3}{4}}\| X^{h_n}(s)-X^h(s)\|_{H}ds.
 \end{eqnarray}

Let $P_k$ be the orthogonal projection in $H$ onto the space spanned by $\{e_1,\cdot\cdot\cdot, e_k\}_{k\geq 1}$ with $\{e_k\}$ be an orthonormal basis of $H$, we have
 \begin{eqnarray*}
&&\sup_{0\leq t\leq T}\|P_kJ^n_3(t)-J^n_3(t)\|^2_H\\
&\leq& \int^1_0\Big(\int^t_0\int^1_0G^2_{t-s}(x,y)\Big((P_k-I)\sigma(u^0(y,s)))^2dyds\Big)\Big(\int^t_0\int^1_0(h_n-h)^2dyds\Big)dx\\
&\leq& 2M^2\int^1_0\int^t_0\int^1_0G^2_{t-s}(x,y)\Big((P_k-I)\sigma(u^0(y,s))\Big)^2dydsdx\\
&\leq& 2M^2\int^t_0(t-s)^{-\frac{1}{2}}\int^1_0\Big((P_k-I)\sigma(u^0(y,s))\Big)^2dyds.
\end{eqnarray*}
Since
\begin{eqnarray*}
2M^2\int^t_0(t-s)^{-\frac{1}{2}}\int^1_0[(P_k-I)\sigma(u^0(y,s))]^2dyds\leq CM^2K^2T^{\frac{1}{2}},
\end{eqnarray*}
by the dominated convergence theorem, it follows that
\begin{eqnarray}\label{equat-3}
\sup_{0\leq t\leq T}\|P_kJ^n_3(t)-J^n_3(t)\|^2_H\rightarrow0, \quad k\rightarrow \infty.
\end{eqnarray}
For any $k\geq 1$, $t,s\in [0,T], t>s$, we have
\begin{eqnarray*}
\|P_kJ^n_3(t)-P_kJ^n_3(s)\|^2_H&\leq& \Big\|\int^s_0\int^1_0(G_{t-r}(x,y)-G_{s-r}(x,y))P_k\sigma(u^0)(h_n(y,r)-h(y,r))dydr\Big\|^2_H\\
&& +\Big\|\int^t_s\int^1_0G_{t-r}(x,y)P_k\sigma(u^0)(h_n(y,r)-h(y,r))dydr\Big\|^2_H\\
&:=& J^n_{3,1}+J^n_{3,2},
\end{eqnarray*}
By (H1) and (\ref{eqq-5-1}), we get
\begin{eqnarray*}
J^n_{3,1}&\leq& \int^1_0\Big(\int^s_0\int^1_0(G_{t-r}(x,y)-G_{s-r}(x,y))^2\sigma^2(u^0(y,r))dydr\Big)\Big(\int^s_0\int^1_0(h_n(y,r)-h(y,r))^2dydr\Big)dx\\
&\leq& CK^2M\int^1_0\int^s_0\int^1_0(G_{t-r}(x,y)-G_{s-r}(x,y))^2dydrdx\\
&\leq& CK^2M(t-s)^{\frac{1}{2}}.
\end{eqnarray*}
Utilizing (\ref{eqq-5}) and (H1), we deduce that
\begin{eqnarray*}
J^n_{3,2}&\leq& \int^1_0\Big(\int^t_s\int^1_0G^2_{t-r}(x,y)\sigma^2(u^0(y,r))dydr\Big)\Big(\int^t_s\int^1_0(h_n(y,r)-h(y,r))^2dydr\Big)dx\\
&\leq& CMK^2\int^1_0\int^t_s\int^1_0G^2_{t-r}(x,y)dydrdx\\
&\leq& CMK^2(t-s)^{\frac{1}{2}}.
\end{eqnarray*}
Combing the above two estimates, it yields
\begin{eqnarray*}
\|P_kJ^n_3(t)-P_kJ^n_3(s)\|^2_H\leq CMK^2(t-s)^{\frac{1}{2}}.
\end{eqnarray*}
Moreover, for any $t\in [0,T]$, we have
\begin{eqnarray*}
\sup_n\|J^n_3(t)\|^2_H&\leq& \sup_n\int^1_0\Big(\int^t_0\int^1_0G^2_{t-s}(x,y)\sigma^2(u^0)dyds\Big)\Big(\int^t_0\int^1_0(h_n(y,s)-h(y,s))^2dyds\Big)dx\\
&\leq& CMK^2\int^1_0\int^t_0\int^1_0G^2_{t-s}(x,y)dydsdx\\
&\leq& CMK^2T^{\frac{1}{2}}.
\end{eqnarray*}
Since for any $k\geq 1$, $P_k:H\rightarrow H$ is a compact operator, then
for any $t\in [0,T]$, $\{P_kJ^n_3(t),n\geq 1 \}$ is relative compact in $H$. Moreover, $\{P_kJ^n_3(t),n\geq 1 \}$ is closed in $H$. As a result of Arzel\`{a}-Ascoli theorem, $\{P_kJ^n_3\}_n$ is uniformly compact in $C([0,T];H)$.
On the other hand, since $h_n-h$ converges to $0$ weakly in $L^2([0,T]\times [0,1]; \mathbb{R})$, then
\[
P_kJ^n_3(t)=\int^t_0\int^1_0G_{t-s}(x,y)P_k\sigma(u^0)(h_n-h)dyds\rightarrow 0\quad {\rm{in}} \ H, \quad {\rm{as}}\ n\rightarrow \infty.
\]
Thus, we have
\begin{eqnarray}\label{ee-2}
\lim_{n\rightarrow \infty}\sup_{t\in [0,T]}\|P_kJ^n_3(t)\|_H=0.
\end{eqnarray}
Combing (\ref{equat-3}) and (\ref{ee-2}), we conclude that
 \begin{eqnarray}\label{ee-3}
 \lim_{n\rightarrow \infty}\sup_{t\in [0,T]}\|J^n_3(t)\|_H= 0.
\end{eqnarray}
Based on (\ref{ee-6}), (\ref{ee-7}) and (\ref{ee-3}), it follows that
\begin{eqnarray*}
\| X^{h_n}(t)-X^h(t)\|_H\leq 2KC(1+C_0)\int^t_0(t-s)^{-\frac{3}{4}}\| X^{h_n}(s)-X^h(s)\|_{H}ds+\|J^n_3(t)\|_H.
\end{eqnarray*}
By iteration and Gronwall inequality, we have
\begin{eqnarray*}
\| X^{h_n}(t)-X^h(t)\|_H\leq C(K, C_0,T) \|J^n_3(t)\|_H.
\end{eqnarray*}
Utilizing (\ref{ee-3}), it yields
\begin{eqnarray*}
\sup_{t\in [0,T]}\| X^{h_n}(t)-X^h(t)\|_H
\leq C(K, C_0,T) \sup_{t\in [0,T]}\|J^n_3(t)\|_H\rightarrow0,\ {\rm{as}}\ n\rightarrow \infty.
\end{eqnarray*}
We complete the proof.
\end{proof}

\begin{thm}\label{thm-4}
Let $\{h^\varepsilon; \varepsilon>0\}\subset \mathcal{A}_M$ be a sequence that converges in distribution to $h$ as $\varepsilon\rightarrow 0$. Then
\[
\mathcal{G}^{\varepsilon}\Big(W(\cdot)+\lambda(\varepsilon)\int^{\cdot}_0\int^{\cdot}_0h^\varepsilon(s,y)dyds\Big)\ {\rm{converges\ in\ distribution\ to}}\ \mathcal{G}^{0}\Big(\int^{\cdot}_0\int^{\cdot}_0h(s,y)dyds\Big),
\]
in $C([0,T];H)$.
\end{thm}

\begin{proof}
Suppose that $\{h^\varepsilon; \varepsilon>0\}\subset \mathcal{A}_M$ and $h^\varepsilon$ converges to $h$ as $T_M-$valued random elements in distribution.
By Girsanov's theorem, we obtain $\bar{X}^{\varepsilon,h^{\varepsilon}}(\cdot)=\mathcal{G}^{\varepsilon}\Big(W(\cdot)+\lambda(\varepsilon)\int^{\cdot}_0\int^{\cdot}_0h^\varepsilon(s,y)dyds\Big)$.
Consider
\begin{eqnarray}
Z^{\varepsilon}(t,x)=\frac{1}{\lambda(\varepsilon)}\int^t_0\int^1_0G_{t-s}(x,y)\sigma(s,y, u^0(s,y)+\sqrt{\varepsilon}\lambda(\varepsilon)\bar{X}^{\varepsilon, h^{\varepsilon}}(s,y))W(dyds),
\end{eqnarray}
with the initial value $Z^{\varepsilon}(0)=0$. Applying the same method as the proof of (\ref{equat-2}), we get
\begin{eqnarray}
\lim_{\varepsilon\rightarrow 0}E\sup_{t\in [0,T]}\|Z^{\varepsilon}(t)\|^2_H=0
\end{eqnarray}
and $\{Z^{\varepsilon}\}$ is tight in $C([0,T];H)$.
Set
\[
\Pi=\Big(C([0,T];H), T_M, C([0,T];H)\Big).
\]
By Proposition \ref{prp-6}, we know that the family $\{(\bar{X}^{\varepsilon, h^{\varepsilon}}, h^{\varepsilon}, Z^{\varepsilon}); \varepsilon\in (0,1]\}$ is tight in $\Pi$. Let $(X,h,0)$ be any limit point of $\{(\bar{X}^{\varepsilon, h^{\varepsilon}}, h^{\varepsilon}, Z^{\varepsilon}); \varepsilon\in (0,1]\}$. We will show that $X$ has the same law as $\mathcal{G}^0(\int^{\cdot}_0\int^{\cdot}_0h(s,y)dyds)$
, and in fact $\bar{X}^{\varepsilon, h^{\varepsilon}}$ converges in distribution to $X$ in $C([0,T];H)$ as $\varepsilon\rightarrow 0$, which implies Theorem \ref{thm-4}.

%
%
By the Skorokhod representation theorem, there exists a stochastic basis $(\Omega^1, \mathcal{F}^1, \{\mathcal{F}^1_t\}_{t\in [0,T]}, {P}^1)$ and $\Pi-$valued random variables $(\tilde{U}^{\varepsilon}, \tilde{h}^{\varepsilon},\tilde{Z}^{\varepsilon}),(\tilde{U}, \tilde{h},0)$ on this basis, such that $(\tilde{U}^{\varepsilon}, \tilde{h}^{\varepsilon},\tilde{Z}^{\varepsilon})$ (resp. $(\tilde{U}, \tilde{h},0)$) has the same law as $(\bar{X}^{\varepsilon,h^{\varepsilon}},h^{\varepsilon},Z^{\varepsilon})$ (resp. $(X,h,0)$), and $(\tilde{U}^{\varepsilon}, \tilde{h}^{\varepsilon},\tilde{Z}^{\varepsilon})\rightarrow (\tilde{U}, \tilde{h},0)$, $P^1-$a.s. in $\Pi$. From the equation satisfied by $(\bar{X}^{\varepsilon,h^{\varepsilon}},h^{\varepsilon},Z^{\varepsilon})$ , we see that $(\tilde{U}^{\varepsilon}, \tilde{h}^{\varepsilon},\tilde{Z}^{\varepsilon})$ satisfies
 \begin{eqnarray}\notag
&&\tilde{U}^{\varepsilon}(x,t)-\tilde{Z}^{\varepsilon}(x,t)\\ \notag
&=& \int^t_0\int^1_0G_{t-s}(x,y)\sigma(s,y, u^0(s,y)+\sqrt{\varepsilon}\lambda(\varepsilon)\tilde{U}^{\varepsilon}(s,y))\tilde{h}^{\varepsilon}(s,y)dyds\\ \notag
&&+\frac{1}{\sqrt{\varepsilon}\lambda(\varepsilon)}\int^t_0\int^1_0G_{t-s}(x,y)\Big(b(s,y,u^0(s,y)+\sqrt{\varepsilon}\lambda(\varepsilon)\tilde{U}^{\varepsilon}(s,y))-b(s,y,u^0(s,y))\Big)dyds\\
\label{ee-8}
&&-\frac{1}{\sqrt{\varepsilon}\lambda(\varepsilon)}\int^t_0\int^1_0\partial_yG_{t-s}(x,y)\Big(g(s,y,u^0(s,y)+\sqrt{\varepsilon}\lambda(\varepsilon)\tilde{U}^{\varepsilon}(s,y))-g(s,y,u^0(s,y))\Big)dyds.
\end{eqnarray}

and
\begin{eqnarray}\notag
&&P^1(\tilde{U}^{\varepsilon}-\tilde{Z}^{\varepsilon}\in C([0,T];H))\\ \notag
&=& P(\bar{X}^{\varepsilon, h^{\varepsilon}}-Z^{\varepsilon}\in C([0,T];H))\\
\label{eee-28}
&=&1.
\end{eqnarray}
Let $\Omega^1_0$ be the subset of  $\Omega^1$ such that $(\tilde{U}^{\varepsilon}, h^{\varepsilon},\tilde{Z}^{\varepsilon})\rightarrow (\tilde{U},\tilde{h},0)$ in $\Pi$, we have $P^1(\Omega^1_0)=1$. For any $\tilde{\omega}\in \Omega^1_0$, we have
\begin{eqnarray}\label{ee-12}
\sup_{t\in [0,T]}\|\tilde{U}^{\varepsilon}(\tilde{\omega},t)-\tilde{U}(\tilde{\omega},t)\|^2_H\rightarrow 0,\ {\rm{as}} \ \varepsilon\rightarrow 0.
\end{eqnarray}
Set $\tilde{\eta}^{\varepsilon}(\tilde{\omega},t)=\tilde{U}^{\varepsilon}(\tilde{\omega},t)-\tilde{Z}^{\varepsilon}(\tilde{\omega},t)$, by (\ref{ee-8}), $\tilde{\eta}^{\varepsilon}(\tilde{\omega},x,t)$ satisfies
\begin{eqnarray*}\notag
&&\tilde{\eta}^{\varepsilon}(\tilde{\omega},x,t)\\ \notag
&=& \int^t_0\int^1_0G_{t-s}(x,y)\sigma(s,y, u^0(s,y)+\sqrt{\varepsilon}\lambda(\varepsilon)(\tilde{\eta}^{\varepsilon}(\tilde{\omega},s,y)+\tilde{Z}^{\varepsilon}(\tilde{\omega},s,y)))\tilde{h}^{\varepsilon}(\tilde{\omega},s,y)dyds\\ \notag
&&+\frac{1}{\sqrt{\varepsilon}\lambda(\varepsilon)}\int^t_0\int^1_0G_{t-s}(x,y)\Big(b(s,y,u^0+\sqrt{\varepsilon}\lambda(\varepsilon)(\tilde{\eta}^{\varepsilon}(\tilde{\omega},s,y)+\tilde{Z}^{\varepsilon}(\tilde{\omega},s,y)))-b(s,y,u^0)\Big)dyds\\
\label{ee-9}
&&-\frac{1}{\sqrt{\varepsilon}\lambda(\varepsilon)}\int^t_0\int^1_0\partial_yG_{t-s}(x,y)\Big(g(s,y,u^0+\sqrt{\varepsilon}\lambda(\varepsilon)(\tilde{\eta}^{\varepsilon}(\tilde{\omega},s,y)+\tilde{Z}^{\varepsilon}(\tilde{\omega},s,y)))-g(s,y,u^0)\Big)dyds.
\end{eqnarray*}
with initial value $\tilde{\eta}^{\varepsilon}(\tilde{\omega},x,0)=0$. Moreover, we deduce from (\ref{eee-28}) that
\begin{eqnarray}\label{eee-29}
\sup_{t\in [0,T]}\|\tilde{\eta}^{\varepsilon}(\tilde{\omega},t)\|_H<\infty.
\end{eqnarray}
Taking into account the following facts
\begin{eqnarray}
\lim_{\varepsilon\rightarrow 0}\sup_{t\in [0,T]}\|\tilde{Z}^{\varepsilon}(\tilde{\omega}, t)\|^2_H=0, \quad \tilde{U}^{\varepsilon}=\tilde{\eta}^{\varepsilon}+\tilde{Z}^{\varepsilon},
\end{eqnarray}
and by (H1), (\ref{eee-29}), (\ref{ee-3}), Lemma \ref{lem-1} and Lemma \ref{lem-5}, we have
\begin{eqnarray}\notag
&&\lim_{\varepsilon\rightarrow 0}\sup_{t\in [0,T]}\|\tilde{U}^{\varepsilon}(\tilde{\omega}, t)-\hat{U}(\tilde{\omega}, t)\|^2_H\\ \notag
&\leq& \lim_{\varepsilon\rightarrow 0}\sup_{t\in [0,T]}\|\tilde{\eta}^{\varepsilon}(\tilde{\omega}, t)-\hat{U}(\tilde{\omega}, t)\|^2_H+\lim_{\varepsilon\rightarrow 0}\sup_{t\in [0,T]}\|\tilde{Z}^{\varepsilon}(\tilde{\omega}, t)\|^2_H\\
\notag
&\leq&\lim_{\varepsilon\rightarrow 0}\sup_{t\in [0,T]}\|\tilde{\eta}^{\varepsilon}(\tilde{\omega}, t)-\hat{U}(\tilde{\omega}, t)\|^2_H\\
\label{ee-11}
&=&0,
\end{eqnarray}
where $\hat{U}(t):=\hat{U}(\tilde{\omega},t)$ satisfies
\begin{eqnarray}\notag
\hat{U}(x,t)&=&\int^t_0\int^1_0G_{t-s}(x,y)\partial_rb(s,y,u^0(s,y))\hat{U}(s,y)dyds\\ \notag
&& -\int^t_0\int^1_0\partial_yG_{t-s}(x,y)\partial_rg(s,y,u^0(s,y))\hat{U}(s,y)dyds\\
\label{ee-10}
&& +\int^t_0\int^1_0G_{t-s}(x,y)\sigma(s,y,u^0(s,y))\tilde{h}(s,y)dyds.
\end{eqnarray}

Hence, by (\ref{ee-12}) and (\ref{ee-11}), we deduce that $\tilde{U}=\hat{U}=\mathcal{G}^0(\int^{\cdot}_0\int^{\cdot}_0\tilde{h}(s,y)dyds)$, then $\tilde{U}$ has the same law as $\mathcal{G}^0(\int^{\cdot}_0\int^{\cdot}_0h(s,y)dyds)$. Since $\bar{X}^{\varepsilon,h^{\varepsilon}}$ has the same law as $\tilde{U}^{\varepsilon}$ on $C([0,T];H)$ and by (\ref{ee-11}), we deduce that $\mathcal{G}^{\varepsilon}(W(\cdot)+\lambda(\varepsilon)\int^{\cdot}_0\int^{\cdot}_0h^{\varepsilon}(s,y)dyds)$ converges in distribution to  $\mathcal{G}^0(\int^{\cdot}_0\int^{\cdot}_0h(s,y)dyds)$
as $\varepsilon\rightarrow 0$. We complete the proof.
\end{proof}


\def\refname{ References}

\end{document}